\documentclass[oneside]{amsart}

\usepackage{amssymb}

\usepackage{enumitem}

\usepackage{graphicx}

\makeatletter
\@namedef{subjclassname@2020}{%
  \textup{2020} Mathematics Subject Classification}
\makeatother
\usepackage{orcidlink}

\usepackage[backend=biber,style=ext-alphabetic,hyperref=true,giveninits=true]{biblatex}
\usepackage{amsmath,amsthm,mathtools,textcomp,extarrows,bm,mleftright}
\mleftright
\usepackage[T1]{fontenc}
\usepackage{mathrsfs}
\DeclareMathAlphabet{\mathpzc}{OT1}{pzc}{m}{it}
\newcommand{\Newt}[1]{\operatorname{\mathpzc{Newt}}\left(#1\right)}
\usepackage{xparse}
\ExplSyntaxOn
\NewDocumentCommand{\definealphabet}{mmmm}
 {%
  \int_step_inline:nnn { `#3 } { `#4 }
   {
    \cs_new_protected:cpx { #1 \char_generate:nn { ##1 }{ 11 } }
     {
      \exp_not:N #2 { \char_generate:nn { ##1 } { 11 } }
     }
   }
 }
\ExplSyntaxOff
\definealphabet{bb}{\mathbb}{A}{Z}
\definealphabet{cal}{\mathcal}{A}{Z}
\definealphabet{frak}{\mathfrak}{A}{z}
\definealphabet{rm}{\mathrm}{A}{z}
\definealphabet{bf}{\mathbf}{A}{Z}
\newcommand{\Supp}[1]{\mathrm{Supp}{\left(#1\right)}}
\newcommand{\zetax}{\zeta_{2(p-1)}}
\DeclareRobustCommand{\stirling}{\genfrac\{\}{0pt}{}}

\newcommand{\newerr}{O\left(p^{1+\frac{2}{p-1}}\right)}
\usepackage{dsfont}

\usepackage{hyperref}
\usepackage[nameinlink]{cleveref}
\Crefname{conjecture}{Conjecture}{Conjectures}
\Crefname{lemma}{Lemma}{Lemmas}
\Crefname{definition}{Definition}{Definitions}
\Crefname{remark}{Remark}{Remarks}
\Crefname{proposition}{Proposition}{Propositions}
\Crefname{corollary}{Corollary}{Corollarys}
\Crefname{equation}{}{}
\Crefname{item}{}{}
\Crefname{algorithm}{Algorithm}{Algorithms}
\usepackage{xcolor}
\usepackage{float,algorithm,algpseudocode}
\newtheorem{theorem}{Theorem}[section]
\newtheorem*{theorem*}{Theorem}
\newtheorem{lemma}[theorem]{Lemma}
\newtheorem{remark}[theorem]{Remark}
\newtheorem{proposition}[theorem]{Proposition}
\newtheorem*{proposition*}{Proposition}
\newtheorem{definition}[theorem]{Definition}

\newtheorem{corollary}[theorem]{Corollary}

\numberwithin{equation}{section}
\numberwithin{figure}{section}
\algnewcommand\algorithmicinput{\textbf{INPUT:}}
\algnewcommand\INPUT{\item[\algorithmicinput]}
\algnewcommand\algorithmicoutput{\textbf{OUTPUT:}}
\algnewcommand\OUTPUT{\item[\algorithmicoutput]}
\newlist{propenum}{enumerate}{1}
\setlist[propenum]{label=(\arabic*), ref=\theproposition~(\arabic*)}
\crefalias{propenumi}{proposition}
\newlist{lemenum}{enumerate}{1}
\setlist[lemenum]{label=(\arabic*), ref=\thelemma~(\arabic*)}
\crefalias{lemenumi}{lemma}
\newlist{thmenum}{enumerate}{1}
\setlist[thmenum]{label=(\arabic*), ref=\thetheorem~(\arabic*)}
\crefalias{thmenumi}{theorem}

\makeatletter
\newcommand{\subalign}[1]{%
  \vcenter{%
    \Let@ \restore@math@cr \default@tag
    \baselineskip\fontdimen10 \scriptfont\tw@
    \advance\baselineskip\fontdimen12 \scriptfont\tw@
    \lineskip\thr@@\fontdimen8 \scriptfont\thr@@
    \lineskiplimit\lineskip
    \ialign{\hfil$\m@th\scriptstyle##$&$\m@th\scriptstyle{}##$\hfil\crcr
      #1\crcr
    }%
  }%
}
\makeatother
\usepackage{microtype,geometry}
\usepackage{breakurl}
\allowdisplaybreaks

\title[Truncated expansion of $\zeta_{p^n}$ in the $p$-adic Mal'cev-Neumann field]{Truncated expansion of $\zeta_{p^n}$ in the $p$-adic Mal'cev-Neumann field}

\author{Shanwen Wang}
\address{School of Mathematics, Renmin University of China\\ No. 59 Zhongguancun Street, Haidian District\\
    Beijing, 100872, China}
\email{s\_wang@ruc.edu.cn}
\thanks{Shanwen Wang is supported by the Fundamental Research Funds for the Central Universities, the Research Funds of Renmin University of China No.2020030251 and The National Natural Science Foundation of China (Grant No.11971035).}

\author{Yijun Yuan\orcidlink{0000-0001-6571-6980}}
\address{Yau Mathematical Sciences Center, Tsinghua University, Haidian District\\
    Beijing, 100084, China}
\email{941201yuan@gmail.com}

\date{}


\begin{document}

\begin{abstract}
    Fix an odd prime $p$. In this article, we provide a $\mathrm{mod}\  p$  harmonic number identity, which appears naturally in the canonical expansion of a root $\zeta_{p^n}$ of the $p^n$-th cyclotomic polynomial $\Phi_{p^n}(T)$ in the $p$-adic Mal'cev-Neumann field $\bbL_p$.  We establish a $\frac{2}{(p-1)p^{n-2}}$-truncated expansion of $\zeta_{p^n}$ via a variant of the transfinite Newton algorithm, which gives the first $\aleph_0^2$ terms of the canonical expansion of $\zeta_{p^n}$. The harmonic number identity simplifies the expression of this expansion.
\end{abstract}

\maketitle

\tableofcontents

\section{Introduction}
Fix an odd prime $p$. Let $\Phi_{p^n}(T)$ be the $p^n$-th cyclotomic polynomial. The harmonic number $H_k$ is the special value of the partial sum $\sum_{n=1}^{k}\frac{1}{n^s}$ of the Riemann $\zeta$ series at $s=1$. It serves as a bridge between the theory of combinatorics and number theory.  In this article, we provide an instant of these phenomenons. More precisely, on one hand, we provide a $\mathrm{mod}\  p$  harmonic number identity using the $r$-restricted Stirling number of the second kind, which also encodes the congruence properites of the $r$-restricted Stirling numbers of the second kind; on the other hand, this harmonic number identity appears naturally in the truncated expansion of $\zeta_{p^n}$ in the $p$-adic Mal'cev-Neumann field $\bbL_p$, and we use it to establish the $\frac{2}{p^{n-2}(p-1)}$-truncated expansion of $\zeta_{p^n}$, which will be used to construct the uniformizer of the false-Tate curve extension of $\bbQ_p$ in \cite{WYInPre}.

\subsection{A $\mathrm{mod}\  p$ identity for harmonic numbers}\label{intro1}
The Bell polynomials are used to study set partitions in combinatorial mathematics. Let $\alpha_l = (j_1, j_2, \cdots, j_l)\in \bbN^{l}$ be a multi-index. We denote its norm by $\vert\alpha_l\vert=j_1 + j_2 + \cdots + j_l$ and its factorial by
$\alpha_l! = \prod_{k=1}^l j_k!$. Let $\bm{x}=(x_1,\cdots, x_l)$ be a $l$-tuple of formal variables. The power of a multi-index $\alpha_l$ of $\bm{x}$ is defined by
$$ \bm{x}^{\alpha_l}:= \prod_{i=1}^lx_i^{j_i}.$$

\begin{definition}
    For integer numbers $n\geq k\geq 0$, the \textbf{incomplete exponential Bell polynomial} $B_{n,k}(x_{1},x_{2},\dots ,x_{n-k+1})$ with parameter $(n,k)$ is a polynomial given by the sum
    \begin{align*}
        \sum_{\substack{\alpha_{n-k+1}=(j_1,\cdots,j_{n-k+1})\in\bbN^{n-k+1} \\\vert\alpha_{n-k+1}\vert=k, \sum_{i=1}^{n-k+1}ij_i=n}} \frac{n!}{\alpha_{n-k+1}!}\left(\frac{x_1}{1!},\cdots, \frac{x_{n-k+1}}{(n-k+1)!}\right)^{\alpha_{n-k+1}}.
    \end{align*}

\end{definition}
With multinomial theorem, the incomplete exponential Bell polynomial can also be defined in terms of its generating function (cf. \cite[P.134 Theorem A]{Comtet1974}):
\begin{equation*}
    \frac{1}{k!}\left(\sum_{m\geq 1}x_m\frac{t^m}{m!}\right)^k=\sum_{n\geq k}B_{n,k}(x_1,\cdots,x_{n-k+1})\frac{t^n}{n!},\ k=0,1,2,\cdots .
\end{equation*}

The special values of the incomplete exponential Bell polynomial at the points $(1,\cdots, 1)$ and $(\overbrace{1,\cdots,1}^{r},0,\cdots,0),$ called Stirling numbers of the second kind and $r$-restricted Stirling numbers of the second kind (cf. \cite{Komatsu2016,Mezo2014}) respectively. More precisely, we have the following definition:
\begin{definition}\leavevmode
    \begin{enumerate}
        \item For integer numbers $n\geq k\geq 0$, the \textbf{Stirling number of the second kind} is defined by
              $$\stirling{n}{k}= B_{n,k}(1,1,\cdots,1) ;$$
        \item For integer numbers $n\geq k\geq 0$ and positive integer $r$, the \textbf{$r$-restricted Stirling number of the second kind} is defined by
              $$\stirling{n}{k}_{\leq r}=
                  \begin{cases}
                      \stirling{n}{k},                                & \text{ if } n-k+1\leq r; \\
                      B_{n,k}(\overbrace{1,\cdots,1}^{r},0,\cdots,0), & \text{ otherwise.}
                  \end{cases}$$
    \end{enumerate}
\end{definition}
Our first result is the following $\mathrm{mod}\  p$ harmonic number identity, which is a translation of the congruence property of the $r$-restricted Stirling numbers of the second kind (cf. \Cref{lem:35904}):
\begin{theorem*}[cf. \Cref{maintheorm1}]
    For an odd prime $p$ and an integer $1\leq k\leq p-1 $, we have
    $$\sum_{i=1}^k\frac{1}{(k-i)!}\left(\sum_{m=1}^p\frac{(p-1)!}{(p-m)!(i+p)!}\stirling{i+p}{m}_{\leq p-1}\right)-\frac{1}{(k-1)!}\equiv -\frac{1}{k!}H_k \bmod{p},$$
    where $H_k$ is the harmonic number.
\end{theorem*}

\subsection{Truncated expansion of $\zeta_{p^2}$ in the $p$-adic Mal'cev-Neumann field}
\subsubsection{The $p$-adic Mal'cev-Neumann field }
Let $\calO_{\breve{\bbQ}_p}=W(\bar{\bbF}_p)$ be the ring of Witt vectors over $\bar{\bbF}_p$ and let $\bbL_p$ be the $p$-adic Mal'cev-Neumann field $\calO_{\breve{\bbQ}_p}((p^{\bbQ}))$ (cf. \cite[Section 4]{Poonen1993}).

Every element $\alpha$ of $\bbL_p$ can be uniquely written as
\begin{equation}\label{expansioncan}\sum_{x\in \bbQ}[\alpha_x]p^x, \text{ where } [\cdot]: \bar{\bbF}_p\rightarrow W(\bar{\bbF}_p) \text{ is the Teichmuller character.}\end{equation}
For any $\alpha=\sum_{x\in \bbQ}[\alpha_x]p^x\in \bbL_p$, we set $ \mathrm{Supp}(\alpha)=\{x\in \bbQ:  \alpha_x\neq 0\}$, which is well-orderd by the definition of $\bbL_p$.
Thus, we can define the $p$-adic valuation $v_p$ by the formula: $$v_p(\alpha)=\begin{cases} \inf \mathrm{Supp}(\alpha), & \text{ if } \alpha\neq 0; \\ \infty, & \text{if } \alpha=0.\end{cases} $$
The field $\bbL_p$ is complete for the $p$-adic topology and it is also algebraically closed. Moreover, it is the maximal complete immediate extension\footnote{A valued field extension $(E,w)$ of  $(F,v)$ is an immediate extension, if  $(E,w)$ and $(F,v)$ have the same residue field. A valued field  $(E,w)$  is maximally complete if it has no immediate extensions other than $(F,v)$ itself.} of $\overline{\bbQ}_p$. The field $\bbL_p$ is spherical complete\footnote{A valued field is said to be spherical complete, if the intersection of every decreasing sequence of closed balls is nonempty.}, and the field $\bbC_p$ of $p$-adic complex numbers is not spherical complete, which can be continuously embedded into $\bbL_p$.
\begin{definition}\leavevmode
    \begin{enumerate}
        \item For any $\alpha\in \bbL_p$, we call the unique expression \Cref{expansioncan} of $\alpha$, the \textbf{canonical expansion} of $\alpha$.
        \item Let $r\in \bbQ$ and $\alpha\in \bbL_p$, we rewrite the canonical expansion of $\alpha$ in the following way
              \[\alpha=\sum_{x\in \bbQ, x< r  }[\alpha_x]p^x+ O(p^r),\]
              and we call this expression the $r$-\textbf{truncated canonical expansion} of $\alpha$.
        \item If $\sum_{x\in \bbQ, x< r  }\beta_xp^x+ O(p^r)$ is another element in $\bbL_p$ with $\beta_x\in \calO_{\breve{\bbQ}_p}$ such that
              \[\sum_{x\in \bbQ, x< r  }\beta_xp^x\equiv \alpha \bmod{p^r},\] then we call $\sum_{x\in \bbQ, x< r  }\beta_xp^x+ O(p^r)$ a $r$-\textbf{truncated expansion} of $\alpha$.
    \end{enumerate}
\end{definition}

Given $\alpha\in \bbL_p$, for $x\in\bbQ$, we denote the coefficient of $p^x$ in the expansion of $\alpha$ by $[C_x(\alpha)]\in \calO_{ \breve{\bbQ}_p}$. Then $C_x(\alpha)$ is the coefficient of $p^x$ modulo $p$. This gives a map
$$ C: \bbQ\times\bbL_p\rightarrow  \bar{\bbF}_p;  (x, \alpha)\mapsto C_x(\alpha) .$$

\subsubsection{Transfinite Newton algorithm}
Let $\Phi_{p^n}(T)=\sum_{k=0}^{p-1}T^{p^{n-1}k}\in\bbQ_p[T]$ be the $p^n$-th cyclotomic polynomial. In \cite{WangYuan2021}, we expanded a result of Kedlaya (cf. \cite[Proposition 1]{Kedlaya2001}) into a transfinite Newton algorithm to study the canonical expansion in $\bbL_p$ of a root of a polynomial $\Phi(T)\in \bbL_p[T]$.

Let $\Lambda_{p-1}= \sum_{k=0}^{p-1}\frac{1}{[k!]}\zetax^k p^{\frac{k}{p(p-1)}}$ and $\sigma_2= \sum_{k=2}^\infty p^{-1/p^k}$.
We have the following explicit formula for the first $\aleph_0$ terms of a root $\zeta_{p^2}$ of the $p^2$-th cyclotomic polynomial $\Phi_{p^2}(T)$, which is a special case of the \cite[Theorem 3.3]{WangYuan2021}.
\begin{theorem}
    Let \(\zeta_{p^2}^{(i)}\) be the \(i\)-th approximation of \(\zeta_{p^2}\) in the transfinite Newton algorithm. Then we have
    \[\zeta_{p^2}^{(i)}=
        \begin{cases}
            \sum_{k=0}^i \frac{1}{[k!]}\zetax^kp^{\frac{k}{p(p-1)}},                 & \text{ for } 0\leq i\leq p-1, \\
            \Lambda_{p-1}+ \sum_{l=2}^{i-p+2}\zetax p^{\frac{1}{p-1}-\frac{1}{p^l}}, & \text{ for }  i\geq p.
        \end{cases}
    \]
    In other words, $\zeta_{p^2}=\Lambda_{p-1}+\zetax p^{\frac{1}{p-1}}\sigma_2+O\left(p^{\frac{1}{p-1}}\right)$.
\end{theorem}
\begin{remark}
    Theoretically, the transfinite Newton algorithm is an ``effective" method to get the canonical expansion of an algebraic number in $\bbL_p$, but the computation is relatively heavy even for finite steps. The notion of $r$-truncated canonical expansion also encodes the information about how many steps one needs to get the $r$-truncated canonical expansion using the transfinite Newton algorithm. In fact, it reduces to count the cardinality of the $r$-truncated support (i.e. the cardinality of the set $\{x\in \bbQ: \alpha_x\neq 0, x< r\}$).

    On the other hand, in practice, it is not obvious that one can get the first $\aleph_0$-terms of the canonical expansion of an algebraic number.
    For the first $\aleph_0$-terms of the canonical expansion of $\zeta_{p^n}$, we did try the first few terms, make a conjecture on the general formula, then prove it inductively. %
\end{remark}

We introduce a variant of transfinite Newton algorithm (cf. \Cref{transfinite}), which allows us to establish the $\left(\frac{1}{p-1}+\frac{1}{p(p-1)}\right)$-truncated canonical expansion of $\zeta_{p^2}$ and the $\frac{2}{p-1}$-truncated expansion of $\zeta_{p^2}$.

\begin{theorem}\label{coro:40822}\leavevmode
    \begin{enumerate}
        \item (cf. \Cref{prop:47112}) $$\zeta_{p^2}=\Lambda_{p-1}+\zeta_{2(p-1)}p^{\frac{1}{(p-1)}}\sigma_2+\zetax^2p^{\frac{1}{p-1}+\frac{1}{p(p-1)}}\sigma_2+O\left(p^{\frac{1}{p-1}+\frac{1}{p(p-1)}}\right).$$
        \item (cf. \Cref{mainexpansion})\begin{align*}
                  \zeta_{p^2}= & \left(\sum_{k=0}^{p-1}\frac{1}{[k!]}\zetax^k p^{\frac{k}{p(p-1)}}\right)\left(1+\zetax p^{\frac{1}{p-1}}\sigma_2\right)+\frac{1}{2}\zetax^2 p^{\frac{2}{p-1}}\sigma_2^2        \\
                               & \quad+\frac{1}{2}\zetax^3p^{\frac{2}{p-1}-\frac{p-2}{p^2(p-1)}}-\sum_{k=1}^{p-1}\frac{H_k}{k!}\zetax^{k+1} p^{\frac{1}{p-1}+\frac{k}{p(p-1)}}+O\left(p^{\frac{2}{p-1}}\right).
              \end{align*}
    \end{enumerate}
\end{theorem}
\begin{remark}
    The first expansion corresponds to the first $2\aleph_0$-terms of the canonical expansion of $\zeta_{p^2}$ and the second expansion corresponds to the first $\aleph_0^2$-terms of the canonical expansion of $\zeta_{p^2}$.
\end{remark}

\section{A $\mathrm{mod}\  p$ identity of harmonic numbers}
\subsection{Stirlings numbers of the second kind}

In \Cref{intro1}, we define the ($r$-restricted) Stirling numbers of the second kind as the special values of the incomplete exponential Bell polynomial and the generating function formula of Bell polynomial implies the generating function formula for the ($r$-restricted) Stirling numbers of the second kind. Combining with the fact that $\stirling{n}{k}_{\leq r}=0$ for $n\geq rk+1$, we can write the generating function formula as
\begin{equation}\label{genfunctform}
    \frac{1}{k!}\left(\sum_{m= 1}^r\frac{t^m}{m!}\right)^k=\sum_{n= k}^\infty\stirling{n}{k}_{\leq r}\frac{t^n}{n!},\ k=0,1,2,\cdots .
\end{equation}

We denote by $(x)_{n}=x(x-1)(x-2)\cdots (x-n+1)$ the falling factorials, which form a basis of the $\bbQ$-vector space $\bbQ[x]$.
The Stirling numbers of the second kind may also be characterized as the coordinate of powers of the indeterminate $x$ with respect to the basis consisting of the falling factorials (cf.  \cite[Page 207 Theorem B]{Comtet1974}) :  If $n>0$, one has
\begin{equation*}
    x^n=\sum_{m=0}^n\stirling{n}{m} (x)_m .
\end{equation*}

The following proposition collects the arithmetic properties of ($r$-restricted) Stirling numbers of the second kind in our previous work.
\begin{proposition}\leavevmode
    \begin{propenum}
        \item \label{it:1}\cite[Corollary 3.12]{WangYuan2021} \begin{equation*}
            \sum_{k=1}^n(-1)^{k-1}(k-1)!\stirling{n}{k}=
            \begin{cases}
                0, & n\geq 2; \\
                1, & n=1.
            \end{cases}
        \end{equation*}
        \item \label{it:2}\cite[Lemma 3.14]{WangYuan2021}Let $p$ be an odd prime number. For an integer $k$ that $1\leq k\leq p$, one has
        \begin{equation*}
            \stirling{p-1+k}{p}\equiv
            \begin{cases}
                1 \bmod{p}, & \text{ if } k=1 \text{ or } p; \\
                0 \bmod{p}, & \text{ otherwise}.
            \end{cases}
        \end{equation*}
        \item \label{it:3}\cite[Lemma 3.15]{WangYuan2021}Let $p$ be an odd prime number and $r$ be an integer number satisfying $1\leq r<p-1$, then one has
        $$\stirling{r+p}{p}_{\leq r}=B_{r+p,p}(1,\cdots,1,0)\equiv 0 \bmod{p} .$$
        \item \label{it:4}\cite[Lemma 3.16]{WangYuan2021}Let $i$ be an integer that $1\leq i\leq p-1$ and $k\in \bbZ_{>0}$. Then for any integer $l\geq k$, we have
        $$v_p\left(\frac{k!}{l!}\stirling{l}{k}_{\leq i}\right)\geq 0.$$
    \end{propenum}
\end{proposition}

\subsection{Congruence properties of Stirling numbers of the second kind }
In this paragraph, we will establish three congruence properties of the ($r$-restricted) Stirling numbers of the second kind.

\begin{lemma}\label{lem:36099}
    Fix a prime $p\geq 3$. For $1\leq i\leq p-1$ and $1\leq m\leq p$, we have
    $$p\mid \stirling{i+p}{m}_{\leq p-1}.$$
\end{lemma}
\begin{proof}
    If $m<p$, by \Cref{it:4}, we have
    $$v_p\left(\frac{m!}{(i+p)!}\stirling{i+p}{m}_{\leq p-1}\right)\geq 0.$$
    As a result, we have
    $$v_p\left(\stirling{i+p}{m}_{\leq p-1}\right)\geq v_p\left((i+p)!\right)-v_p(m!)=1.$$
    If $m=p$, then by \cite[Proposition 1]{Komatsu2016}
    \begin{align*}
        \stirling{i+p}{p}_{\leq p-1}= & p\stirling{i+p-1}{p}_{\leq p-1}+\stirling{i+p-1}{p-1}_{\leq p-1}-\binom{i+p-1}{p-1}\stirling{i}{p-1}_{\leq p-1} \\
        \equiv                        & \stirling{i+p-1}{p-1}_{\leq p-1}-\binom{i+p-1}{p-1}\stirling{i}{p-1}_{\leq p-1} \bmod{p}.
    \end{align*}
    Again by \Cref{it:4}, we have
    $$v_p\left(\frac{(p-1)!}{(i+p-1)!}\stirling{i+p-1}{p-1}_{\leq p-1}\right)\geq 0.$$
    Therefore
    $$v_p\left(\stirling{i+p-1}{p-1}_{\leq p-1}\right)\geq v_p\left((i+p-1)!\right)-v_p((p-1)!)=1.$$
    Notice that
    $$v_p\left(\binom{i+p-1}{p-1}\right)=v_p((i+p-1)!)-v_p(i!)-v_p((p-1)!)=1,$$
    one can conclude that $p\mid \stirling{i+p}{m}_{\leq p-1}$.

\end{proof}

\begin{lemma}\label{lem:52893}
    For a prime $p\geq 3$ and an integer $1\leq i\leq p-1$, we have
    $$\sum_{m=p+1}^{i+p} (-1)^{m-1}(m-1)!\stirling{i+p}{m}_{\leq p-1}=\begin{cases}p+O\left(p^2\right),&\ i=1\\O\left(p^2\right),&\ i\geq 2\end{cases}.$$
\end{lemma}
\begin{proof}We rewrite the summation in the lemma as follows:
    \begin{equation*}
        \begin{split}
            \sum_{m=p+1}^{i+p} (-1)^{m-1}(m-1)!\stirling{i+p}{m}_{\leq p-1}
            =\sum_{t=1}^{i} (-1)^t(t+p-1)!\stirling{i+p}{t+p}_{\leq p-1}.
        \end{split}
    \end{equation*}
    Thus, we reduce to study the $p$-adic valuation of each term.
    \begin{enumerate}
        \item
              Since $(p-1)!\equiv -1\bmod{p}$,  we have
              $$(t+p-1)! =(p-1)!\cdot p\cdot (p+1)\cdots (p+t-1)\equiv -p(t-1)!\bmod{p^2}.$$
        \item As $(i+p)-(t+p)+1=i-t+1\leq p-1$, we have $\stirling{i+p}{t+p}_{\leq p-1}=\stirling{i+p}{t+p}$. By \cite[Corollary 2.1]{Sanchez2000}, we have
              $$\stirling{i+p}{t+p}\equiv \stirling{i+1}{t+p}+\stirling{i}{t} \bmod{p}.$$ Since $i+1<t+p$, we have $\stirling{i+1}{t+p}=0$.
    \end{enumerate} By combining these estimations, we have
    \begin{align*}
          & \sum_{m=p+1}^{i+p} (-1)^{m-1}(m-1)!\stirling{i+p}{m}_{\leq p-1}                               \\
        = & \sum_{t=1}^{i} (-1)^t\left(-p(t-1)!+O\left(p^2\right)\right)\left(\stirling{i}{t}+O(p)\right) \\
        = & O\left(p^2\right)+p\sum_{t=1}^i (-1)^{t-1}(t-1)!\stirling{i}{t}                               \\
        = & \begin{cases}p+O\left(p^2\right),&\ i=1\\O\left(p^2\right),&\ i\geq 2\end{cases},
    \end{align*}
    where the last equality follows from \Cref{it:1}.

\end{proof}

Let $$B_n(x_1,\cdots,x_n)=\sum_{k=1}^n B_{n,k}(x_1,\cdots,x_{n-k+1}), n=1,2,\cdots$$
be the $n$-th complete Bell polynomial. For $n\geq 1$, we define a map
$$\bbB_n: \bbZ^n\rightarrow\bbZ^n, (x_1,\cdots,x_n)\mapsto (B_1(x_1),B_2(x_1,x_2),\cdots,B_n(x_1,\cdots,x_n)).$$

\begin{proposition}\leavevmode Let $p\geq 3$ be a prime.
    \begin{propenum}
        \item For $n\in \bbN_{\geq 1}$, $\bbB_n$ is bijective and we denote its inverse simply by $\bbB_n^{-1}$.
        \item \label{it:x}Let $\delta_j=\begin{cases}1,&\ j\leq p-1,\\0,&\ j\geq p\end{cases}$ and $1\leq i\leq p-1$ be an integer. We have
        $$\bbB_{p+i}^{-1}(\delta_1,\cdots,\delta_{p+i})=(1,0\cdots,0,-1,\frakx_{p+1},\cdots,\frakx_{p+i}),$$
        where $\frakx_{p+k}\in \bbZ$ satisfies $\frakx_{p+k}\equiv \frac{(-1)^{k+1}p}{k}\bmod{p^2}$ for $1\leq k\leq i$.
    \end{propenum}
\end{proposition}

\begin{proof}
    \begin{enumerate}
        \item We prove the first assertion by induction on $n$. For $n=1$, we have $\bbB_1(x)=x$. Assume for $1\leq k\leq n$, $\bbB_k$ are bijective. We will prove the assertion for $k=n+1$.
              For any $(y_1, \cdots, y_{n+1})\in \bbZ^{n+1}$, there exists a unique $(x_1,\cdots, x_n)\in \bbZ^n$ such that $\bbB_n(x_1, \cdots, x_n)=(y_1,\cdots,y_n)$.   The recurrence relation\footnote{Cited from \cite[(12.3.4)]{Andrews1984}. We regard $B_0$ as $1$.}
              $$B_{n+1}(x_{1},\ldots ,x_{n+1})=\sum_{i=0}^{n}\binom{n}{i}B_{n-i}(x_{1},\ldots ,x_{n-i})x_{i+1},$$
              shows that $\bbB_{n+1}$ is bijective. In fact, we can set
              \begin{equation}\label{eq:50448}\begin{aligned}
                      x_{n+1}= & y_{n+1}-\sum_{i=0}^{n-1}\binom{n}{i}B_{n-i}(x_{1},\ldots ,x_{n-i})x_{i+1} \\
                      =        & y_{n+1}-\sum_{i=0}^{n-1}\binom{n}{i}y_{n-i}x_{i+1}.
                  \end{aligned}\end{equation}

        \item For $p=3$, the result follows from a direct calculation. Therefore, in the following, we assume $p\geq 5$. Let $(\frakx_1,\cdots,\frakx_{p+i}):=\bbB_{p+i}^{-1}(\delta_1,\cdots,\delta_{p+i})$. By the proof of the first assertion, we calculate the components of $(\frakx_1,\cdots,\frakx_{p+i})$ inductively. We have $\frakx_1=1$. By the formula \Cref{eq:50448}, we have
              $$\frakx_2=\delta_2-\sum_{i=0}^0\binom{1}{i}\delta_{1-i}\frakx_{i+1}=1-B_1(\frakx_1)=0.$$
              For $2\leq l\leq p-2$, assume that $\frakx_2,\cdots,\frakl_l=0$, we calculate
              $$\frakx_{l+1}=\delta_{l+1}-\sum_{i=0}^{l-1}\binom{l}{i}\delta_{l-i}\frakx_{i+1}=1-B_1(\frakx_1)=0.$$
              For $\frakx_p$ and $\frakx_{p+1}$, a direct calculation shows that
              \begin{equation*}
                  \begin{split}
                      \frakx_p&=\delta_{p}-\sum_{i=0}^{p-2}\binom{p-1}{i}\delta_{p-i-1}\frakx_{i+1}=0-B_1(\frakx_1)=-1,\\
                      \frakx_{p+1}&=\delta_{p+1}-\sum_{i=0}^{p-1}\binom{p}{i}\delta_{p-i}\frakx_{i+1}=0-\binom{p}{p-1}\frakx_p=p.
                  \end{split}
              \end{equation*}
              It rests to show the $\mathrm{mod}\ p^2$ values of $\frakx_{p+i}$ for $1\leq i\leq p-1$, which we will show by induction. Suppose $\frakx_{p+t}\equiv \frac{(-1)^{t+1}p}{t}\bmod{p^2}$ for $1\leq t\leq k\leq p-2$. Thus, we have the following identity
              \begin{equation}\label{identitywy}
                  \begin{split}
                      \frakx_{p+k+1}= & \delta_{p+k+1}-\sum_{i=0}^{p+k-1}\binom{p+k}{i}\delta_{p+k-i}\frakx_{i+1}
                      =               -\sum_{i=p-1}^{p+k-1}\binom{p+k}{i}\frakx_{i+1}                                                                 \\
                      =               & -\left[\binom{p+k}{p-1}\cdot (-1)+\sum_{t=1}^k\binom{p+k}{p+t-1}\frac{(-1)^{t+1}p}{t}+O\left(p^2\right)\right].
                  \end{split}
              \end{equation}
              For $1\leq t\leq k$, we have the following estimations:
              \begin{equation*}
                  \begin{split}
                      \binom{p+k}{p-1}&=p\frac{(p+1)\cdots(p+k)}{(k+1)!}=p\frac{k!+O(p)}{(k+1)!}=\frac{p}{k+1}+O\left(p^2\right),\\
                      \frac{1}{t}\binom{p+k}{p+t-1}&=\frac{(p+t)\cdots(p+k)}{(k-t+1)!\cdot t}=\frac{(t)\cdots (k)+O(p)}{(k-t+1)!\cdot t}\\
                      &=O(p)+\frac{1}{k+1}\cdot\frac{(k+1)!}{(k+1-t)!t!}=O(p)+\frac{1}{k+1}\binom{k+1}{t}.
                  \end{split}
              \end{equation*}
              Thus, the identity \Cref{identitywy} can be written as following
              \begin{equation*}
                  \begin{split}
                      \frakx_{p+k+1}&=\frac{p}{k+1}+\frac{p}{k+1}\sum_{t=1}^k \binom{k+1}{t}(-1)^t+O\left(p^2\right)\\
                      =&\frac{p}{k+1}\sum_{t=0}^{k+1} \binom{k+1}{t}(-1)^t-\frac{(-1)^{k+1}p}{k+1}+O\left(p^2\right)\\
                      =&\frac{(-1)^{k+2}p}{k+1}+O\left(p^2\right),
                  \end{split}
              \end{equation*}
              which is the expected formula.
    \end{enumerate}
\end{proof}

\begin{proposition}[Inverse relations of Bell polynomials]
    For an integer $n\geq 1$ and $y_1,\cdots,y_n\in\bbZ$, we have $\bbB_n^{-1}(y_1,\cdots,y_n)=(y_1^\vee,\cdots,y_n^\vee)$, where
    $$y_i^\vee=\sum_{k=1}^i(-1)^{k-1}(k-1)!B_{i,k}(y_1,\ldots ,y_{i-k+1}),\ i=1,\cdots,n.$$
\end{proposition}
\begin{proof}
    See \cite[(50) of Chapter 2]{Riordan1980}.
\end{proof}

Using this proposition for $n=p+i$ , $(y_1,\cdots,y_{p+i})=(\delta_1,\cdots,\delta_{p+i})$ and comparing with \Cref{it:x}, we can conclude:
\begin{proposition}\label{coro:38801}
    For a prime $p\geq 3$ and an integer $1\leq i\leq p-1$, we have
    $$\sum_{m=1}^{i+p} (-1)^{m-1}(m-1)!\stirling{i+p}{m}_{\leq p-1}\equiv \frac{(-1)^{i+1}p}{i}\bmod{p^2}.$$
\end{proposition}
\subsection{Harmonic numbers and the Stirling numbers of the second kind}
In this paragraph, we prove the following $\mathrm{mod}\  p$ harmonic number identity:
\begin{theorem}\label{maintheorm1}
    For a prime $p\geq 3$ and an integer $1\leq k\leq p-1 $, we have
    $$\sum_{i=1}^k\frac{1}{(k-i)!}\left(\sum_{m=1}^p\frac{(p-1)!}{(p-m)!(i+p)!}\stirling{i+p}{m}_{\leq p-1}\right)-\frac{1}{(k-1)!}\equiv -\frac{1}{k!}H_k \bmod{p},$$
    where $H_k$ is the harmonic number.
\end{theorem}
This theorem can be deduced from the following proposition:
\begin{proposition}\label{lem:35904}
    For a prime $p\geq 3$ and an integer $1\leq i\leq p-1$, we have
    $$\sum_{m=1}^p\frac{(p-1)!}{(p-m)!(i+p)!}\stirling{i+p}{m}_{\leq p-1}\equiv \begin{cases}0,& i=1\\\frac{(-1)^i}{i!\cdot i},& 2\leq i\leq p-1\end{cases} \bmod{p}.$$
\end{proposition}
\begin{proof}
    For $p=3$, the result follows from direct calculation. Therefore we assume $p\geq 5$. To prove the proposition, it is equivalent to prove the congruence
    \begin{equation}\label{wang1}\frac{i!\cdot i}{(i+p)!}\sum_{m=1}^p\frac{(p-1)!}{(p-m)!}\stirling{i+p}{m}_{\leq p-1}\equiv \begin{cases}0,&\ i=1\\(-1)^i,&\ i\geq 2\end{cases} \bmod{p}.\end{equation}
    Notice that
    \begin{align*}
        \frac{i!}{(i+p)!}= & \frac{i!}{(p-1)!\cdot p\cdot (p+1)\cdots(p+i)}       \\
        =                  & \frac{i!}{\left(-1+O(p)\right)p\left(i!+O(p)\right)}
        =                   -\frac{1}{p}+O(1)
    \end{align*}
    and $$\frac{(p-1)!}{(p-m)!}\equiv (-1)^{m-1}(m-1)! \bmod{p}.$$ By \Cref{lem:36099} we can reduce the congruence \Cref{wang1} to
    $$-\frac{i}{p}\sum_{m=1}^p (-1)^{m-1}(m-1)!\stirling{i+p}{m}_{\leq p-1}\equiv \begin{cases}0,&\ i=1\\(-1)^i,&\ i\geq 2\end{cases} \bmod{p}.$$
    By \Cref{lem:52893}, we have
    \begin{align*}
               & \sum_{m=1}^p (-1)^{m-1}(m-1)!\stirling{i+p}{m}_{\leq p-1}                                                                                                                                                      \\ =&\sum_{m=1}^{i+p} (-1)^{m-1}(m-1)!\stirling{i+p}{m}_{\leq p-1}-\sum_{m=p+1}^{i+p} (-1)^{m-1}(m-1)!\stirling{i+p}{m}_{\leq p-1} \\
        \equiv & \begin{cases}\displaystyle\sum_{m=1}^{1+p} (-1)^{m-1}(m-1)!\stirling{1+p}{m}_{\leq p-1}-p,&\ i=1\\\displaystyle\sum_{m=1}^{i+p} (-1)^{m-1}(m-1)!\stirling{i+p}{m}_{\leq p-1},&\ i\geq 2\end{cases} \bmod{p^2}.
    \end{align*}
    Then the result follows from \Cref{coro:38801}.
\end{proof}
\begin{proof}[Proof of \Cref{maintheorm1}]
    If $k=1$, the theorem is trivial. For $k\geq 2$,  by \Cref{lem:35904}, we have
    \begin{align*}
               & \sum_{i=1}^k\frac{1}{(k-i)!}\left(\sum_{m=1}^p\frac{(p-1)!}{(p-m)!(i+p)!}\stirling{i+p}{m}_{\leq p-1}\right)-\frac{1}{(k-1)!} \\
        \equiv & \sum_{i=2}^k \frac{1}{(k-i)!}\frac{(-1)^i}{i!\cdot i}-\frac{1}{(k-1)!}\bmod{p}                                                \\ =&\frac{1}{k!}\sum_{i=1}^k \binom{k}{i}\frac{(-1)^i}{i}=-\frac{1}{k!}H_k,
    \end{align*}
    where the last equality follows from \cite[(9.2)]{Boyadzhiev2018}.
\end{proof}

\section{Truncated expansion of $\zeta_{p^2}$: a reflection on transfinite Newton algorithm}
Let $P(T)=a_0T^n+a_1 T^{n-1}+\cdots+ a_n\in \bbL_p[T] $ be a polynomial with $a_n\neq 0$.
\begin{definition}For any $\mu\in\bbL_p$, we call $P_{\mu}(T)=P(T+ \mu)$ the $\mu$-perturbed polynomial of $P(T)$.\end{definition}
Let $P_{\mu}(T)=\sum_{k=0}^n=b_kT^{n-k}$. Then for $0\leq k\leq n$, we have $$b_k=\sum_{j=0}^{k}a_{k-j}\binom{n-k+j}{j}\mu^j.$$
\subsection{Transfinite Newton algorithm  and Newton polygon}
In \cite{WangYuan2021}, we expanded a result of Kedlaya (cf. \cite[Proposition 1]{Kedlaya2001}) into a transfinite Newton algorithm to find the canonical expansion of a root of the polynomial $P(T)\in \bbL_p[T]$ in the $p$-adic Mal'cev-Neumann field $\bbL_p$.
\begin{definition}[Newton polygon]\leavevmode Suppose $(K,v)$ is a valued field with value group $\bbQ$. Let $J(T)=\sum_{i=0}^n a_{n-i}T^i\in K[T]$ be a nonzero polynomial. For $0\leq i\leq n$, we have the points $(i, v(a_{i}))\in \bbN\times \bar{\bbR}$, where $\bar{\bbR}= \bbR\cup\{+\infty\}$. If $a_i=0$, $(i,v(a_i))$ is regarded as $Y_{+\infty}$, the point at infinity of the positive vertical axis.
    \begin{enumerate}
        \item Define the \textbf{Newton polygon} $\Newt{J}$ of $J(T)$ as the lower boundary of the convex hull of the points $(i,v(a_i))$ for $i=0,\cdots,n$. As a consequence, $\Newt{J}$ is a function on $\bbR_{\geq 0}$ with values in $\bar{\bbR}$.
        \item The integers $m$ such that $(m,v(a_m))$ are vertices of $\Newt{J}$ are called the \textbf{breakpoints}, and we denote by $m_{\max}^J$ the \textbf{largest breakpoint} less than $n$.
        \item Given two adjacent breakpoints $m^J_1<m^J_2$, denote by $s^J_{m_1}=\frac{v(a_{m^J_2})-v(a_{m^J_1})}{m^J_2-m^J_1}$, the \textbf{slope} of constituent segment of $\Newt{J}$ with endpoints $(m^J_1,v(a_{m^J_1}))$ and $(m^J_2,v(a_{m^J_2}))$. The \textbf{largest slope} is denoted by $$s_{\max}^J= s_{m_{\max}^J}^J=\frac{v(a_n)-v(a_{m_{\max}^J})}{n-m_{\max}^J}.$$ If $(n,v(a_n))=Y_{+\infty}$ (i.e. $a_n=0$)\footnote{Notice that if $m$ is a breakpoint, then $(m,v(a_m))=Y_{+\infty}\Leftrightarrow m=n$ and $a_n=0$.}, we regard $s_{\max}^J=\infty$. Thus, $s_{\max}^\bullet$ is a map from $K[T]$ to $\bbQ\cup\{\infty\}$.
    \end{enumerate}
    We will omit the superscript $J$ if there is no confusion.
\end{definition}
Apply the above notions to $P(T)=a_0T^n+a_1 T^{n-1}+\cdots+ a_n\in \bbL_p[T] $ and we set $s=s_{\max}^P$ and $m=m_{\max}^{P}$.
\begin{definition}
    We define the \textit{residue polynomial} associated to $P(T)$ by
    \[\mathrm{Res}_P(T)=\sum_{k=0}^{n-m}C_0\left(a_{n-k}p^{-v_p(a_{m})-s(n-m-k)}\right)T^k \in \bar{\bbF}_p[T] .\]
\end{definition}
The transfinite Newton algorithm can be summarized in the following pseudo-code:
\begin{algorithm}[H]
    \caption{transfinite Newton algorithm for \(\bbL_p\)}    \begin{algorithmic}
        \INPUT A non-constant polynomial \(P(T)\in \bbL_p[T]\)
        \OUTPUT A root of \(P(T)\) in \(\bbL_p\)
        \Function{Newton}{$P$}
        \State \(r\gets 0\)
        \State \(s_{\max}\gets 0,m_{\max}\gets 0,c\gets 0\)
        \State \(\mathrm{Res}_{\Phi}(T)\gets 0\)
        \State \(\Phi(T)\gets f(T)\) \Comment{We denote the coefficient of \(T^i\) in \(\Phi\) as \(b_{n-i}\),where \(n=\text{deg}(\Phi)\).}
        \While{\(\Phi(0)\neq 0\)}
        \State \(m_{\max}\gets m_{\max}^{\Phi}\)
        \State \(s_{\max}\gets s_{\max}^{\Phi}\)
        \State \(\mathrm{Res}_{\Phi}(T)\gets\sum_{k=0}^{n-m_{\max}}C_{v_p(b_m)+s_{\max}(n-m_{\max}-k)}\left(b_{n-k}\right)T^k\)
        \State \(c\gets\) any root of \(\mathrm{Res}_{\Phi}(T)\) in \(\bar{\bbF}_p\)
        \State \(r\gets r+[c]\cdot p^{s_{\max}}\)
        \State \(\Phi(T)\gets \Phi(T+[c]\cdot p^{s_{\max}})\)
        \EndWhile
        \State \Return \(r\)
        \EndFunction
    \end{algorithmic}
\end{algorithm}

\begin{definition}
    We call the value of $\Phi(T)$ (resp. $\mathrm{Res}_{\Phi}(T)$ and $r$) in the above pseudo-code after the loop iterates for $l$ times, the $l$-th approximation polynomial (resp. the $l$-th residue polynomial and the $l$-th approximation of a root of $\Phi(T)$). In particular, the $l$-th approximation polynomial is the $\mu$-perturbed polynomial by taking $\mu$ to be the $l$-th approximation of a root of $\Phi(T)$.

\end{definition}

\subsection{A variant of transfinite Newton algorithm}\label{transfinite}
In this paragraph, we will introduce a variant of transfinite Newton algorithm and sketch the proof of the \Cref{coro:40822} via this new method.
\subsubsection{$\left(1+\frac{2}{p-1}\right)$-truncated expansion of $\Lambda_{p-1}^p-1$}
Let $\lambda=\zetax p^{\frac{1}{p(p-1)}}$ and $\Lambda_{p-1}=\sum_{k=0}^{p-1}\frac{\lambda^k}{[k!]}$. We establish the $\left(1+\frac{2}{p-1}\right)$-truncated expansion of $\Lambda_{p-1}^p-1$ (cf. \Cref{lem:43810}), which will be used to give the $p\aleph_0$-terms of the truncated expansion of $\zeta_{p^2}$.

Let $\widehat{\Lambda}_{p-1}=\sum_{k=0}^{p-1}\frac{\lambda^k}{k!}$. Since $\widehat{\Lambda}_{p-1}-\Lambda_{p-1}\in O\left(p^{1+\frac{2}{p(p-1)}}\right)$, we have
\[\Lambda_{p-1}^p-1= \widehat{\Lambda}_{p-1}^p-1+O\left(p^2\right).\]
Using binomial expansion, we can write
\begin{equation}\label{eq:64747}
    \widehat{\Lambda}_{p-1}^p-1= \left(\sum_{k=1}^{p-1}\frac{\lambda^k}{k!}\right)^p+\sum_{j=1}^{p-1}\binom{p}{j}\left(\sum_{k=1}^{p-1}\frac{\lambda^k}{k!}\right)^j.
\end{equation}
Thus, we reduce to study the $\left(1+\frac{2}{p-1}\right)$-truncated expansions of $\left(\sum_{k=1}^{p-1}\frac{\lambda^k}{k!}\right)^p$ and $\sum_{j=1}^{p-1}\binom{p}{j}\left(\sum_{k=1}^{p-1}\frac{\lambda^k}{k!}\right)^j$, which will be established in \Cref{termp} and \Cref{termj} respectively.

\begin{lemma}\label{termp}
    The $\left(1+\frac{2}{p-1}\right)$-truncated expansion for $\left(\sum_{k=1}^{p-1}\frac{\lambda^k}{k!}\right)^p$ is :
    $$\left(\sum_{k=1}^{p-1}\frac{\lambda^k}{k!}\right)^p=\sum_{k=1}^{p-1}\frac{\lambda^{kp}}{k!}+\sum_{n=1}^{p-1}\frac{p!}{(n+p)!}\stirling{n+p}{p}_{\leq p-1}\lambda^{n+p}+\newerr.$$
\end{lemma}
\begin{proof}
    For the term $ \left(\sum_{k=1}^{p-1}\frac{\lambda^k}{k!}\right)^p$, by the generating formula (\ref{genfunctform}), we have
    $$\left(\sum_{k=1}^{p-1}\frac{\lambda^k}{k!}\right)^p=\sum_{n=p}^{p(p-1)}\frac{p!}{n!}\stirling{n}{p}_{\leq p-1}\lambda^n.$$
    Let $\delta_j=\begin{cases}1,&\text{ if } j\leq p-1,\\0,&\text{ if } j\geq p,\end{cases}$ and we have
    \begin{align*}
        \frac{p!}{n!}\stirling{n}{p}_{\leq p-1}
        =  \sum_{\substack{j_1,\cdots,j_{n-p+1}\geq 0 \\\sum j_k=p,\sum kj_k=n}}\binom{p}{j_1,\cdots,j_{n-p+1}}\left(\frac{\delta_1}{1!}\right)^{j_1}\cdot\cdots\cdot \left(\frac{\delta_{n-p+1}}{(n-p+1)!}\right)^{j_{n-p+1}}.
    \end{align*}
    If $p\nmid n$, then $j_1,\cdots,j_{n-p+1}<p$ and $p\mid \binom{p}{j_1,\cdots,j_{n-p+1}}$. Therefore, if $p\nmid n$ and $n\geq 2p$, together with $v_p(\lambda)=\frac{1}{p(p-1)}$, we have $v_p\left(\frac{p!}{n!}\stirling{n}{p}_{\leq p-1}\lambda^n\right)\geq 1+\frac{2}{p-1}$.
    As a consequence, we have:
    \begin{equation*}
        \begin{split}
            &  \sum_{n=p}^{p(p-1)}\frac{p!}{n!}\stirling{n}{p}_{\leq p-1}\lambda^n\\
            =&\sum_{k=1}^{p-1}\frac{p!}{(kp)!}\stirling{kp}{p}_{\leq p-1}\lambda^{kp}+\sum_{n=1}^{p-1}\frac{p!}{(n+p)!}\stirling{n+p}{p}_{\leq p-1}\lambda^{n+p}+\newerr.
        \end{split}
    \end{equation*}
    By \cite[Lemma 3.17]{WangYuan2021}, one has
    $$\frac{p!}{(kp)!}\stirling{kp}{p}_{\leq p-1}=\begin{cases}1,&\text{ if } k=1;\\\frac{1}{k!}\stirling{k}{1}_{\leq p-1}+O(p)=\frac{1}{k!}+O(p),&\text{ if } 2\leq k\leq p-1.\end{cases}$$
    Then we have
    $\sum_{k=1}^{p-1}\frac{p!}{(kp)!}\stirling{kp}{p}_{\leq p-1}\lambda^{kp}=\sum_{k=1}^{p-1}\frac{\lambda^{kp}}{k!}+O\left(p^{1+\frac{2}{p-1}}\right)$, and the result follows.
\end{proof}

\begin{lemma}
    Let $p\geq 3$ be a prime. Let $\alpha\in\calO_{\bbL_p}$ and $U_p=1+\frac{1}{(p-1)!}$. We have:
    \begin{lemenum}
        \item \label{it:a} $\sum_{s=1}^{p-1}\binom{p}{s}\sum_{t=s}^{p-1} \left(\frac{s!}{t!}\stirling{t}{s}_{\leq p-1}\right)\alpha^t=p\alpha+O\left(p^{2+2v_p(\alpha)}\right).$
        \item \label{it:b}$\frac{1}{p!}\sum_{m=1}^{p-1}\binom{p}{m}m!\stirling{p}{m}_{\leq p-1}=-U_p+O\left(p^{p-1}\right).$
    \end{lemenum}
\end{lemma}
\begin{proof}
    \begin{enumerate}
        \item Since $t-s+1\leq p-1$, one has $\stirling{t}{s}_{\leq p-1}=\stirling{t}{s}$. Therefore by exchanging the order of the summations, we obtain
              \begin{align*}
                  \sum_{s=1}^{p-1}\binom{p}{s}\sum_{t=s}^{p-1} \left(\frac{s!}{t!}\stirling{t}{s}_{\leq p-1}\right)\alpha^t= \sum_{t=1}^{p-1}\frac{\alpha^t}{t!}\sum_{s=1}^t\binom{p}{s}s!\stirling{t}{s}
                  =                                                                                                         \sum_{t=1}^{p-1}\frac{\alpha^t}{t!}\sum_{s=1}^t (p)_s\stirling{t}{s}.
              \end{align*}
              By \cite[(3.7)]{WangYuan2021}, we have
              $$\sum_{s=1}^t (p)_s\stirling{t}{s}=\sum_{s=0}^t (p)_s\stirling{t}{s}-(p)_0\stirling{t}{0}=p^t.$$
              Therefore, we get:
              $$\sum_{s=1}^{p-1}\binom{p}{s}\sum_{t=s}^{p-1} \left(\frac{s!}{t!}\stirling{t}{s}_{\leq p-1}\right)\alpha^t=\sum_{t=1}^{p-1}\frac{\alpha^t}{t!}p^t=p\alpha+O\left(p^{2+2v_p(\alpha)}\right).$$
        \item Notice that $\stirling{p}{m}_{\leq p-1}=\stirling{p}{m}$ for $m\geq 2$, one can write
              \begin{align*}
                  \frac{1}{p!}\sum_{m=1}^{p-1}\binom{p}{m}m!\stirling{p}{m}_{\leq p-1}
                  =                                                                  \frac{1}{p!}\left(\sum_{m=0}^p (p)_m\stirling{p}{m}-p-p!\right).
              \end{align*}
              Since $\sum_{m=0}^p (p)_m\stirling{p}{m}=p^p$, we have
              $$\frac{1}{p!}\sum_{m=1}^{p-1}\binom{p}{m}m!\stirling{p}{m}_{\leq p-1}=\frac{1}{p!}\left(p^p-p-p!\right)=-U_p+O\left(p^{p-1}\right).$$
    \end{enumerate}
    \phantom{}\qedhere
\end{proof}
\begin{lemma}\label{termj}
    The $\left(1+\frac{2}{p-1}\right)$-truncated  expansion of $\sum_{j=1}^{p-1}\binom{p}{j}\left(\sum_{l=1}^{p-1}\frac{\lambda^l}{l!}\right)^j$ is the following:
    \begin{equation*}
        p\lambda+\zetax p^{\frac{1}{p-1}}\cdot U_p+ \sum_{n=1}^{p-1}\frac{\lambda^{n+p}}{(n+p)!}\left(\sum_{s=1}^{p-1}\binom{p}{s}s!\stirling{n+p}{s}_{\leq p-1}\right)+\newerr,
    \end{equation*}
    where $U_p=1+\frac{1}{(p-1)!}$.
\end{lemma}
\begin{proof}
    Using the generating function formula (\ref{genfunctform}), we  have
    \begin{align*}
        \sum_{j=1}^{p-1}\binom{p}{j}\left(\sum_{l=1}^{p-1}\frac{\lambda^l}{l!}\right)^j
        =                                                                                \sum_{s=1}^{p-1}\binom{p}{s}\sum_{t=s}^\infty \left(\frac{s!}{t!}\stirling{t}{s}_{\leq p-1}\right)\lambda^t.
    \end{align*}
    Since $v_p(\frac{s!}{t!}\stirling{t}{s}_{\leq p-1})\geq 0$ and $v_p(\lambda)=\frac{1}{p(p-1)}$, we have:
    \begin{align*}
          & \sum_{j=1}^{p-1}\binom{p}{j}\left(\sum_{l=1}^{p-1}\frac{\lambda^l}{l!}\right)^j
        =  \sum_{s=1}^{p-1}\binom{p}{s}\sum_{t=s}^{2p-1} \left(\frac{s!}{t!}\stirling{t}{s}_{\leq p-1}\right)\lambda^t+\newerr                                                                                                               \\
        = & \sum_{s=1}^{p-1}\binom{p}{s}\sum_{t=s}^{p-1} \left(\frac{s!}{t!}\stirling{t}{s}_{\leq p-1}\right)\lambda^t+\sum_{s=1}^{p-1}\binom{p}{s}\sum_{t=p}^{2p-1} \left(\frac{s!}{t!}\stirling{t}{s}_{\leq p-1}\right)\lambda^t +\newerr.
    \end{align*}
    We estimate $ \sum_{s=1}^{p-1}\binom{p}{s}\sum_{t=s}^{p-1} \left(\frac{s!}{t!}\stirling{t}{s}_{\leq p-1}\right)\lambda^t$ and $\sum_{s=1}^{p-1}\binom{p}{s}\sum_{t=p}^{2p-1} \left(\frac{s!}{t!}\stirling{t}{s}_{\leq p-1}\right)\lambda^t$ separately:
    \begin{enumerate}
        \item     By applying \Cref{it:a} to $\alpha=\lambda$, we obtain $$\sum_{s=1}^{p-1}\binom{p}{s}\sum_{t=s}^{p-1} \left(\frac{s!}{t!}\stirling{t}{s}_{\leq p-1}\right)\lambda^t=p\lambda+O\left(p^2\right).$$
        \item By exchanging the order of summations in $\sum_{s=1}^{p-1}\binom{p}{s}\sum_{t=p}^{2p-1} \left(\frac{s!}{t!}\stirling{t}{s}_{\leq p-1}\right)\lambda^t$, we have
              $$\sum_{s=1}^{p-1}\binom{p}{s}\sum_{t=p}^{2p-1} \left(\frac{s!}{t!}\stirling{t}{s}_{\leq p-1}\right)\lambda^t=\sum_{t=p}^{2p-1}\frac{\lambda^t}{t!}\left(\sum_{s=1}^{p-1}\binom{p}{s}s!\stirling{t}{s}_{\leq p-1}\right).$$
              Moreover, for $t=p$, by \Cref{it:b}, we have $$\frac{\lambda^p}{p!}\sum_{s=1}^{p-1}\binom{p}{s}s!\stirling{p}{s}_{\leq p-1}=\zetax p^{\frac{1}{p-1}}U_p+O\left(p^2\right).$$
    \end{enumerate}
    Therefore, by expanding the product, rearranging terms, truncating and shifting the index, we have
    \begin{equation*}
        \begin{split} &\sum_{j=1}^{p-1}\binom{p}{j}\left(\sum_{l=1}^{p-1}\frac{\lambda^l}{l!}\right)^j\\
            =&p\lambda+\zetax p^{\frac{1}{p-1}}\cdot U_p+ \sum_{n=1}^{p-1}\frac{\lambda^{n+p}}{(n+p)!}\left(\sum_{s=1}^{p-1}\binom{p}{s}s!\stirling{n+p}{s}_{\leq p-1}\right)+\newerr .
        \end{split}
    \end{equation*}
    \qedhere
\end{proof}

Let $\tilde{\Lambda}_{p-1}=\sum_{k=0}^{p-1}\frac{\lambda^{kp}}{[k!]}$, $\widehat{\tilde{\Lambda}}_{p-1}=\sum_{k=0}^{p-1}\frac{\lambda^{kp}}{k!}$ and
$\tilde{\Lambda}_{p-1}^+=  \tilde{\Lambda}_{p-1}+\zetax p^{\frac{1}{p-1}}U_p$.

\begin{proposition}\label{lem:43810}
    We have the $\left(1+\frac{2}{p-1}\right)$-truncated expansion of $\Lambda_{p-1}^p-1$:
    \begin{align*}
        \Lambda_{p-1}^p-1= & \widehat{\tilde{\Lambda}}_{p-1}-1+\zetax p^{1+\frac{1}{p(p-1)}}+\zetax p^{\frac{1}{p-1}}\cdot U_p  \\
                           & +\sum_{n=2}^{p-1}\frac{(-1)^{n+1}}{n!n} \zetax^{n+1}p^{1+\frac{1}{p-1}+\frac{n}{p(p-1)}} +\newerr.
    \end{align*}

\end{proposition}
\begin{proof}
    By combinining \Cref{eq:64747}, \Cref{termp} and \Cref{termj} and rearranging terms,  one has
    \begin{align*}
        \widehat{\Lambda}_{p-1}^p-1= & \widehat{\tilde{\Lambda}}_{p-1}-1+\zetax p^{1+\frac{1}{p(p-1)}}+\zetax p^{\frac{1}{p-1}}\cdot U_p                                \\
                                     & \quad+\sum_{n=1}^{p-1}\frac{\lambda^{n+p}}{(n+p)!}\left(\sum_{s=1}^{p}\binom{p}{s}s!\stirling{n+p}{s}_{\leq p-1}\right)+\newerr.
    \end{align*}
    The result follows from \Cref{lem:35904}.
\end{proof}
\begin{remark} The $\mathrm{mod}\ p$ harmonic identity will not appear directly in the computation, since we have used \Cref{lem:35904} to simplify the formula.
\end{remark}
\begin{lemma}\label{lem:2239}
    One has
    \begin{align*}
        \frac{1}{\Lambda_{p-1}^p}= & \sum_{k=0}^{p-1} \frac{\zetax^k}{k!}p^{\frac{k}{p-1}}-\zetax p^{1+\frac{1}{p(p-1)}}-\zetax p^{\frac{1}{p-1}}U_p \\
                                   & \quad-2\zetax^2 p^{1+\frac{1}{p(p-1)}+\frac{1}{p-1}}+\kappa+\newerr,
    \end{align*}
    where $\kappa= -\sum_{n=2}^{p-1}\frac{(-1)^{n+1}}{n!n}\zetax^{n+1}p^{1+\frac{1}{p-1}+\frac{n}{p(p-1)}}.$
\end{lemma}
\begin{proof}

    Since $v_p\left(1-\Lambda_{p-1}^p\right)=\frac{1}{p-1}$, we have
    \begin{align*}
        \frac{1}{\Lambda_{p-1}^p}= & \frac{1}{1-\left(1-\Lambda_{p-1}^p\right)}
        =  \sum_{k=0}^p (-1)^k\left(\Lambda_{p-1}^p-1\right)^k+\newerr                                                                                                               \\
        =                          & \sum_{k=0}^p (-1)^k \left(\widehat{\tilde{\Lambda}}_{p-1}-1+\zetax p^{1+\frac{1}{p(p-1)}}+\zetax p^{\frac{1}{p-1}}U_p-\kappa\right)^k +\newerr.
    \end{align*}
    A direct computation gives
    \begin{align*}
          & \left(\widehat{\tilde{\Lambda}}_{p-1}-1+\zetax p^{1+\frac{1}{p(p-1)}}+\zetax p^{\frac{1}{p-1}}U_p-\kappa\right)^k                                                                                                                                                                              \\
        = & \begin{cases}\left(\widehat{\tilde{\Lambda}}_{p-1}-1\right)^k+\newerr, & \text{ if } k\geq 3; \\ \left(\widehat{\tilde{\Lambda}}_{p-1}-1\right)^2-2\zetax^2 p^{1+\frac{1}{p(p-1)}+\frac{1}{p-1}}+\newerr, & \text{ if } k=2.
            \end{cases}\end{align*}
    Therefore, we have
    \begin{align*}
        \frac{1}{\Lambda_{p-1}^p}= & \sum_{k=0}^p (-1)^k \left(\widehat{\tilde{\Lambda}}_{p-1}-1\right)^k -\zetax p^{1+\frac{1}{p(p-1)}}-\zetax p^{\frac{1}{p-1}}U_p \\
                                   & \quad -2\zetax^2 p^{1+\frac{1}{p(p-1)}+\frac{1}{p-1}} +\kappa+\newerr.
    \end{align*}
    In the following, we expand the term $(*):=\sum_{k=0}^p (-1)^k \left(\widehat{\tilde{\Lambda}}_{p-1}-1\right)^k$, which will allow us to conclude the result.

    Let $\eta= -\zetax p^{\frac{1}{p-1}}$. Using \Cref{it:4}, we obtain
    \begin{align*}
        \sum_{k=0}^p (-1)^k \left(\widehat{\tilde{\Lambda}}_{p-1}-1\right)^k
        =                                                                      1+\sum_{k=1}^p (-1)^k \sum_{n=k}^p\frac{k!}{n!}\stirling{n}{k}_{\leq p-1}\eta^n +\newerr.
    \end{align*}
    By exchanging the order of summations, one has
    \begin{align*}
        (*)= 1+\sum_{n=1}^p \frac{\eta^n}{n!}\sum_{k=1}^n (-1)^k k! \stirling{n}{k}_{\leq p-1} +\newerr.
    \end{align*}
    Together with the fact $\stirling{n}{k}_{\leq p-1}=\stirling{n}{k}$ for $n-k+1\leq p-1$, we have
    \begin{equation*}
        \begin{split}
            (*)= & 1+\sum_{n=1}^p \frac{\eta^n}{n!}\sum_{k=1}^n (-1)^k k! \stirling{n}{k}                               +\frac{\eta^p}{p!}(-1)^1 1! \left(\stirling{p}{1}_{\leq p-1}-\stirling{p}{1}\right) +\newerr \\
            =                                                                     & 1+\frac{\eta^p}{p!}+\sum_{n=1}^p \frac{\eta^n}{n!}\sum_{k=1}^n (-1)^k k! \stirling{n}{k}+\newerr.
        \end{split}
    \end{equation*}
    As
    $\left(p^3-1\right)_k=\left(p^3-1\right)\left(p^3-2\right)\cdots \left(p^3-k\right)=(-1)^k k!+O\left(p^3\right),$
    one has
    \begin{align*}
        \sum_{k=1}^n (-1)^k k! \stirling{n}{k}= & \sum_{k=1}^n\left(\left(p^3-1\right)_k+O\left(p^3\right)\right) \stirling{n}{k}
        =                                       \sum_{k=1}^n\left(p^3-1\right)_k \stirling{n}{k}+O\left(p^3\right)                \\
        =                                       & \left(p^3-1\right)^n+O\left(p^3\right)
        =                                        (-1)^n+O\left(p^3\right),
    \end{align*}
    where the second to last identity is \cite[Page 207 Theorem B]{Comtet1974}.
    Therefore, we have
    $$(*)=1+\frac{\eta^p}{p!}+\sum_{n=1}^p\frac{(-\eta)^n}{n!}+\newerr.$$
    As a consequence, we have
    \begin{align*}
        \frac{1}{\Lambda_{p-1}^p}= & -\zetax p^{1+\frac{1}{p(p-1)}}-2\zetax^2 p^{1+\frac{1}{p(p-1)}+\frac{1}{p-1}}-\zetax p^{\frac{1}{p-1}}U_p+\kappa \\
                                   & \quad +1+\frac{\eta^p}{p!}+\sum_{n=1}^p \frac{(-\eta)^n}{n!}+\newerr                                             \\
        =                          & \sum_{k=0}^{p-1} \frac{\zetax^k}{k!}p^{\frac{k}{p-1}}-\zetax p^{1+\frac{1}{p(p-1)}}-\zetax p^{\frac{1}{p-1}}U_p  \\
                                   & \quad-2\zetax^2 p^{1+\frac{1}{p(p-1)}+\frac{1}{p-1}}+\kappa+\newerr.
    \end{align*}
\end{proof}

\subsubsection{Perturbed polynomial and its Newton polygon}

The first step of our strategy is to use the transfinite Newton algorithm to get the first few terms (finite or infinite) of the canonical expansion of a root of $P(T)$. In the case $P(T)=\Phi_{p^2}(T)$, the first $\aleph_0$ terms of a root $\zeta_{p^2}$ of $P(T)$ is given in \cite[Theorem 3.3]{WangYuan2021}. More precisely, the first $\aleph_0$-terms of $\zeta_{p^2}$ or equivalently the $\left(\frac{1}{p-1}\right)$-truncated canonical expansion of $\zeta_{p^2}$ is the following:
\[\zeta_{p^2}=\Lambda_{p-1}+\zeta_{2(p-1)}p^{\frac{1}{p-1}}\sigma_2 + O\left(p^{\frac{1}{p-1}}\right).\]

The second step is to disturb $P(T)$. Suppose $\mu_0\in\bbL_p$ is an approximation of a root of $P(T)$ obtained in the first step of our strategy. We set $r_0=\sup\mathrm{Supp}(\mu_0)$ and choose a rational number $r>r_0$. We choose randomly an element $\alpha\in \bbL_p$ such that $r_0<v_p(\alpha)$ and $\sup\mathrm{Supp}(\alpha)=r$. Then we disturb $P(T)$ by $\mu=\mu_0+\alpha$ and estimate the maximal slope $s_{\mu}$ of the perturbed polynomial $P_{\mu}(T)$. Since $P_{\mu}(T)=\sum_{k=0}^nb_kT^{n-k}$ with $b_k=\sum_{j=0}^{k}a_{k-j}\binom{n-k+j}{j}\mu^j$ for $0\leq k\leq n$, we can reduce the estimation of the maximal slope to calculate the $p$-adic valuation of $b_k$ for $0\leq k\leq n$. There are two different cases depending on the value of $s_\mu$:
\begin{enumerate}
    \item if $s_{\mu}\geq r$, then $\mu$ is a $r$-truncated expansion of a root of $P(T)$;
    \item if $s_{\mu}<r_0$, then we add an indeterminate $\calM$ to $\mu$. We need to find the value of $\calM$ such that $\mu+\calM$ gives the $r$-truncated expansion of a root of $P(T)$.
          Since $\mu_0$ is an approximation, we have $v_p(\calM)\geq r_0$ and $\mu+\calM= \mu_0+O(p^r)$. This produces a truncated equation by considering $(\mu+\calM)^{p^k}=(\mu_0+O(p^r))^{p^k}$ for some $k\geq 1$. In general, this may not produce a useful equation for arbitrary $P(T)$. But for $\Phi_{p^2}(T)$, this do produce a useful equation.
\end{enumerate}

In the rest of this paragraph, we establish two truncated expansions of a root $\zeta_{p^2}$ of $P(T)=\Phi_{p^2}(T)$, which exactly correspond to the two situations described as above. For any $\mu\in \bbL_p$, we have the following formula for the coefficients of $\mu$-perturbed polynomial
$P_{\mu}(T)=\Phi_{p^2}(T+\mu)\eqqcolon \sum_{k=0}^{p(p-1)}b_{p(p-1)-k}T^k$:
\begin{equation}\label{coefficient}
    \begin{split}
        b_{p(p-1)}&=\sum_{l=0}^{p-1}\mu^{lp}=\frac{\mu^{p^2}-1}{\mu^p-1}; \\
        b_{p(p-1)-1}&=\sum_{l=1}^{p-1}\binom{pl}{1}\mu^{pl-1}=\frac{p}{\mu}\left(\frac{p\mu^{p^2}}{\mu^p-1}-\mu^p\frac{\mu^{p^2}-1}{\left(\mu^p-1\right)^2}\right).
    \end{split}\end{equation}

\begin{enumerate}[leftmargin=\itemindent]
    \item Take $\mu_0=\Lambda_{p-1}+\zeta_{2(p-1)}p^{\frac{1}{p-1}}\sigma_2$, $r_0=\frac{1}{p-1}$, $r=\frac{1}{p-1}+\frac{1}{p(p-1)}$ and
          \[\mu=\Lambda_{p-1}+\left(1+\zetax p^{\frac{1}{p(p-1)}}\right)\zetax p^{\frac{1}{p-1}}\sigma_2.\]
          With the expansion of $\mu^p-1$ and $\mu^{p^2}-1$ (cf. \Cref{lem:23648} and \Cref{eq:31503}), a direct calculation using the strong triangle inequality gives:
          \begin{align*}
              v_p\left(b_{p(p-1)-1}\right) & =v_p\left(\frac{p}{\mu}\left(\frac{p\mu^{p^2}}{\mu^p-1}-\mu^p\frac{\mu^{p^2}-1}{\left(\mu^p-1\right)^2}\right)\right)=2-\frac{1}{p-1}; \\
              v_p\left(b_{p(p-1)}\right)   & =v_p\left(\frac{\mu^{p^2}-1}{\mu^p-1}\right)=2+\frac{1}{p(p-1)}.
          \end{align*}
          Notice that the line passing through the points $\left(p(p-1)-1,v_p\left(b_{p(p-1)-1}\right)\right)$ and $\left(p(p-1),v_p\left(b_{p(p-1)}\right)\right)$ has slope equal to or smaller than the maximal slope of $\Newt{P_\mu}$.
          As a consequence, the maximal slope of $\Newt{P_\mu}$ is
          \[\geq \left(2+\frac{1}{p(p-1)}\right)-\left(2-\frac{1}{p-1}\right)=\frac{1}{p(p-1)}+\frac{1}{p-1},\] which implies that there exists a root of $P_\mu$ with valuation $\geq \frac{1}{p(p-1)}+\frac{1}{p-1}$. Therefore we conclude:
          \begin{proposition}\label{prop:47112}
              There exists a $p^2$-th primitive root of unity $\zeta_{p^2}$ satisfying
              \begin{equation}
                  \zeta_{p^2}=\mu+O\left(p^{\frac{1}{p(p-1)}+\frac{1}{p-1}}\right).
              \end{equation}
          \end{proposition}

    \item Recall that $\sigma_2=\sum_{k=2}^{\infty}p^{-\frac{1}{p^k}}$. We take
          \[\mu_0=\Lambda_{p-1}+\left(1+\lambda\right)\zetax p^{\frac{1}{p-1}}\sigma_2,r_0=\frac{1}{p-1}+\frac{1}{p(p-1)}, r=\frac{2}{p-1}\] and
          $\mu=\Lambda_{p-1}\left(1+\zetax p^{\frac{1}{p-1}}\sigma_2\right)$.
          Suppose $\Lambda=\mu+\calM$ is a $r$-truncated expansion of $\zeta_{p^2}$ with an indeterminate $\calM$. Then
          we have
          \[\Lambda=\mu+\calM=\mu_0+O\left(p^{r_0}\right).\]
          Consider the truncated expansions of $\Lambda^p$.
          \begin{enumerate}
              \item Using $\Lambda=\mu_0+O\left(p^{r_0}\right)$, we get
                    \[\Lambda^p=   \mu_0^p+\sum_{k=1}^{p-1}\binom{p}{k}\mu_0^{p-k}O\left(p^{r_0}\right)^k+O\left(p^{r_0}\right)^p.\]
                    Note that $r_0\cdot p=1+\frac{2}{p-1}$ and $v_p\left(\sum_{k=1}^{p-1}\binom{p}{k}\mu_0^{p-k}O\left(p^{r_0}\right)^k\right)\geq 1+\frac{1}{p-1}+\frac{1}{p(p-1)}$.
                    As $\Lambda$ is an approximation of $\zeta_{p^2}$, the majoration of the maximal slope of the Newton polygon of $P_{\Lambda}$ will improve the truncated position of $\Lambda^p$ (cf.\Cref{lem:26652}):
                    \begin{equation}\label{cote1}\Lambda^p=\tilde{\Lambda}_{p-1}^++O\left(p^{1+r}\right),\end{equation}
                    where $ \tilde{\Lambda}_{p-1}^+= \sum_{l=0}^{p-1}\frac{(-1)^l}{[l!]}\zetax^l p^{\frac{l}{p-1}}+\zetax p^{\frac{1}{p-1}}U_p$.
              \item Using $\Lambda=\mu+\calM$, $v_p(\calM)\geq r_0$ and $v_p(\calM^k)\geq \frac{2}{p-1}$ for $k\geq 2$, we have
                    \[\Lambda^p= \sum_{k=0}^p\binom{p}{k}\mu^{p-k}\calM^k=\mu^p+p\mu^{p-1}\calM+O\left(p^{1+r}\right).\]
                    Thus, together with \Cref{cote1}, we have
                    $\mu^p+p\mu^{p-1}\calM=\tilde{\Lambda}_{p-1}^++O\left(p^{1+r}\right)$.
                    In other words, we have an equation  $p\calM=\mu\left(\frac{\tilde{\Lambda}_{p-1}^+}{\mu^p}-1\right)+O\left(p^{1+r}\right)$.
                    We calculate the $(1+r)$-truncated expansion of $\mu\left(\frac{\tilde{\Lambda}_{p-1}^+}{\mu^p}-1\right)$ in \Cref{truncated3} and obtain the value of $\calM$
                    (cf. \Cref{lem:38120}):
                    One has
                    \begin{align*}
                        \calM= & \frac{1}{2}\zetax^2 p^{\frac{2}{p-1}\sigma_2^2}+\frac{1}{2}\zetax^3p^{\frac{2}{p-1}-\frac{p-2}{p^2(p-1)}}             \\
                               & \quad-\sum_{k=1}^{p-1}\frac{H_k}{k!}\zetax^{k+1} p^{\frac{1}{p-1}+\frac{k}{p(p-1)}} +O\left(p^{\frac{2}{p-1}}\right),
                    \end{align*}
                    where $H_k=\sum_{n=1}^{k}\frac{1}{n}$ is the harmonic number.
          \end{enumerate}

\end{enumerate}

\subsection{Estimations}
Let $P(T)\in \bbL_p[T]$ be a given polynomial. For any $\Lambda\in\bbL_p$, the $\Lambda$-perturbed polynomial $P_{\Lambda}(T)$, in particular the Newton polygon associated to $P_{\Lambda}(T)$,  plays an important role in the transfinite Newton algorithm and its variant. For $P(T)=\Phi_{p^2}(T)$ the $p^2$-th cyclotomic polynomial, the coefficients of its $\Lambda$-perturbed polynomial is determined by $\Lambda^p$ and $\Lambda^{p^2}$ (cf. the formula \Cref{coefficient}).  In this paragraph, we will establish the truncated expansions of $\Lambda^p$ and $\Lambda^{p^2}$ for $$\Lambda=\Lambda_{p-1}+\left(1+\lambda\right)\zetax p^{\frac{1}{p-1}}\sigma_2 \text{ or } \Lambda=\Lambda_{p-1}\left(1+\zeta_{2(p-1)}p^{\frac{1}{p-1}}\sigma_2\right)+\calM.$$

\subsubsection{Two useful propositions}
The first proposition is the following, which will be used in our calculation, in particular, for $A=\sigma_n:=\sum_{k=n}^{\infty}p^{-1/p^k}$.
\begin{proposition}\label{prop:16960}
    Let $A=\sum_{q\in\bbQ}\left[\alpha_q\right]p^q\in\bbL_p$ with $\Supp{A}\subset \left[0,\frac{1}{p}\right)$. Then
    $$A^p=\sum_{q\in \Supp{A}}\left[\alpha_q^p\right]p^{pq}+O\left(p^{1+pv_p(A)}\right).$$
\end{proposition}
\begin{proof}
    For any $q\in\bbQ$, let $\calS_q=\left\{(g_1,\cdots,g_p)\in\Supp{A}^p\middle\vert g_1+\cdots+g_p=q\right\}$ which is a finite set and $T_q=\sum_{(g_1,\cdots,g_p)\in\calS_q}\left[\alpha_{g_1}\cdots\alpha_{g_p}\right]$ the coefficient of degree $q$ term of $A^p$. Note that if $\calS_{q}=\varnothing$, then $T_q=0$.

    For $\underline{\alpha}, \underline{\beta}\in \calS_q$, if $\underline{\alpha}$ equal to $\underline{\beta}$ up to a permutation, then we say $\underline{\alpha}$ is equivalent to $\underline{\beta}$, denoted by $\underline{\alpha}\sim\underline{\beta}$.
    Thus we have
    \[T_q=\sum_{\subalign{&\overline{(q_1, \cdots q_p)}\in S_q/\sim,\\&q_1=\cdots=q_{i_1}\\< &q_{i_1+1}=\cdots=q_{i_2}\\< &\cdots\\< &q_{i_k+1}=\cdots=q_{p}}}\left[\alpha_{q_1}\cdots \alpha_{q_p}\right]\binom{p}{i_1, i_2-i_1,\cdots, p-i_k}.\]
    Note that $v_p\left(\binom{p}{i_1, i_2-i_1, \cdots, p-i_k}\right)<1$ if and only if  $q_1=q_2=\cdots=q_p=\frac{q}{p}$.
    Consequently, one has $T_q= \begin{cases}
            \left[\alpha_{q/p}^p\right]+O(p), & \text{ if }\frac{q}{p}\in \Supp{A}; \\
            O(p),                             & \text{ otherwise}.\end{cases}$

    As a result,
    $$A^p=\sum_{q/p\in \Supp{A}}\left(\left[\alpha_{q/p}^p\right]+O(p)\right)\cdot p^q+\sum_{\calS_q\neq\varnothing \text{ and } q/p\notin \Supp{A}}O(p)\cdot p^q.$$
    Since $\calS_q=\varnothing$ for $q< pv_p(A)$, one obtains
    $$A^p=\sum_{q/p\in \Supp{A}}\left[\alpha_{q/p}^p\right]\cdot p^q+O\left(p^{1+pv_p(A)}\right)=\sum_{q\in \Supp{A}}\left[\alpha_q^p\right]p^{pq}+O\left(p^{1+pv_p(A)}\right).$$
\end{proof}

\begin{corollary}\label{lemma:12551}
    For positive integer number $m,n\geq 1$, we have
    $$\left(\sum_{k=n+m}^\infty p^{-1/p^k}\right)^{p^n}=\sum_{k=m}^\infty p^{-1/p^k}+O\left(p^{1-1/p^m}\right).$$
\end{corollary}
\begin{proof}
    We prove the statement by induction.
    If $n=1$, by \Cref{prop:16960}, we have
    \begin{align*}
        p^{1/p^m}\left(\sum_{k=m+1}^\infty p^{-1/p^k}\right)^p=  \left(\sum_{k=m+1}^\infty p^{1/p^{m+1}-1/p^k}\right)^p
        =                                                        \sum_{k=m+1}^\infty p^{1/p^m-1/p^{k-1}}+O(p).
    \end{align*}
    Then by inverting $p^{1/p^m}$ on both sides, one obtains the result for $n=1$.
    Suppose the statement is true for $1,2,\cdots,n$. Then
    \begin{align*}
        \left(\sum_{k=n+1+m}^\infty p^{-1/p^k}\right)^{p^{n+1}}= & \left(\sum_{k=m+1}^\infty p^{-1/p^k}+O\left(p^{1-1/p^{m+1}}\right)\right)^p                                                  \\
        =                                                        & \sum_{t=0}^p\binom{p}{t}\left(\sum_{k=m+1}^\infty p^{-1/p^k}\right)^{p-t}\cdot \left(O\left(p^{1-1/p^{m+1}}\right)\right)^t.
    \end{align*}
    By moving terms with valuation greater than $1$ into the error term, we have
    \begin{align*}
        \left(\sum_{k=n+1+m}^\infty p^{-1/p^k}\right)^{p^{n+1}}=  \left(\sum_{k=m+1}^\infty p^{-1/p^k}\right)^p+O\left(p^1\right) =        \sum_{k=m}^\infty p^{-1/p^{k-1}}+O\left(p^{1-1/p^m}\right).
    \end{align*}
    The result follows from the induction.
\end{proof}
In the rest, we will establish the second useful proposition (cf. \Cref{prop:55108}), which allows us to reduce the study of $\Lambda_{p-1}^{p^2}$ to that of $\Lambda_{p-1}^p$.

\begin{lemma}\label{lem:13884}
    $$\widehat{\tilde{\Lambda}}_{p-1}^p-1=-\zetax p^{1+\frac{1}{p-1}}\cdot U_p+O\left(p^{2+\frac{2}{p-1}}\right).$$
\end{lemma}
\begin{proof}
    Let \(\eta=-\zetax p^{\frac{1}{p-1}}\). We have
    \begin{equation}\label{eq:51692}
        \left(\sum_{l=0}^{p-1}\frac{\eta^l}{l!}\right)^p-1=\sum_{j=1}^p\binom{p}{j}\left(\sum_{l=1}^{p-1}\frac{\eta^l}{l!}\right)^j=\sum_{j=1}^p\binom{p}{j}\sum_{k=j}^\infty \frac{j!}{k!}\stirling{k}{j}_{\leq p-1}\eta^k.
    \end{equation}
    Notice that
    \[v_p\left(\binom{p}{j}\left(\frac{j!}{k!}\stirling{k}{j}_{\leq p-1}\right)\eta^k\right)\geq 1-v_p(j)+\frac{k}{p-1} ,\]
    by rearranging terms and assembling terms with valuation  \(\geq 2+\frac{2}{p-1}\), we can rewrite \Cref{eq:51692} as
    \begin{align*}
        \left(\sum_{l=0}^{p-1}\frac{\eta^l}{l!}\right)^p-1= & \sum_{j=1}^{p-1}\binom{p}{j}\sum_{k=j}^{p-1} \frac{j!}{k!}\stirling{k}{j}_{\leq p-1}\eta^k+\sum_{j=1}^{p-1}\binom{p}{j}\frac{j!}{p!}\stirling{p}{j}_{\leq p-1}\eta^p \\
                                                            & \quad+\sum_{k=p}^{2p-1}\frac{p!}{k!}\stirling{k}{p}_{\leq p-1}\eta^k+O\left(p^{2+\frac{2}{p-1}}\right).
    \end{align*}
    By \Cref{it:a}, one has
    $$\sum_{j=1}^{p-1}\binom{p}{j}\sum_{k=j}^{p-1} \frac{j!}{k!}\stirling{k}{j}_{\leq p-1}\eta^k=-\zetax p^{1+\frac{1}{p-1}}+O\left(p^{2+\frac{2}{p-1}}\right)$$
    and
    $$\frac{\eta^p}{p!}\sum_{j=1}^{p-1}\binom{p}{j}j!\stirling{p}{j}_{\leq p-1}=-\zetax p^{1+\frac{1}{p-1}}\cdot U_p+O\left(p^3\right).$$
    On the other hand, since $k-p+1\leq p-1$ for $k\leq 2p-2$, we have
    $$\sum_{k=p}^{2p-1}\frac{p!}{k!}\stirling{k}{p}_{\leq p-1}\eta^k=\sum_{k=p}^{2p-2}\frac{p!}{k!}\stirling{k}{p}\eta^k+\frac{p!}{(2p-1)!}\stirling{2p-1}{p}_{\leq p-1}.$$
    Since  \(v_p(p!)=v_p(k!)=1\) for \(k=p,\cdots 2p-2\), by \Cref{it:2} we have
    $$\frac{p!}{k!}\stirling{k}{p}=\begin{cases}O(p),&\text{ if }p<k\leq 2p-2 ;\\1,&\text{ if }k=p.\end{cases}$$
    Besides that, by \Cref{it:3} one has $v_p\left(\frac{p!}{(2p-1)!}\right)=0$ and $$v_p\left(\stirling{2p-1}{p}_{\leq p-1}\right)\geq 1.$$ Therefore
    $$\sum_{k=p}^{2p-1}\frac{p!}{k!}\stirling{k}{p}_{\leq p-1}\eta^k=\eta^p+\sum_{k=p+1}^{2p-2}O\left(p^{1+\frac{k}{p-1}}\right)+O\left(p^{1+\frac{2p-1}{p-1}}\right)=\eta^p+O\left(p^{2+\frac{2}{p-1}}\right).$$

    To sum up, we obtain
    $
        \left(\sum_{l=0}^{p-1}\frac{\eta^l}{l!}\right)^p-1        =                                                    -\zetax p^{1+\frac{1}{p-1}}\cdot U_p+O\left(p^{2+\frac{2}{p-1}}\right).$

\end{proof}
\begin{proposition}\label{prop:55108}
    Let $A\in\bbL_p$. If
    $A=\tilde{\Lambda}_{p-1}^++O\left(p^{1+\frac{1}{p-1}}\right)$,
    then
    $$A^p-1=p\left(A-\tilde{\Lambda}_{p-1}^+\right)+O\left(p^{1+\frac{2}{p-1}}\right).$$
\end{proposition}
\begin{proof}
    Since $\widehat{\tilde{\Lambda}}_{p-1}-\tilde{\Lambda}_{p-1}\in O\left(p^{1+\frac{2}{p-1}}\right)$, we have
    \begin{align*}
        A^p-1= & \left(\widehat{\tilde{\Lambda}}_{p-1}+\zetax p^{\frac{1}{p-1}}U_p+A-\tilde{\Lambda}_{p-1}^++O\left(p^{1+\frac{2}{p-1}}\right)\right)^p-1                                                                                                                                                                                                      \\
        =      & \widehat{\tilde{\Lambda}}_{p-1}^p-1                                                                                                                                     +\sum_{t=1}^p\binom{p}{t}\widehat{\tilde{\Lambda}}_{p-1}^{p-t}\left(\zetax p^{\frac{1}{p-1}}U_p+A-\tilde{\Lambda}_{p-1}^++O\left(p^{1+\frac{2}{p-1}}\right)\right)^t.
    \end{align*}
    By \Cref{lem:13884}, we know that
    \[\widehat{\tilde{\Lambda}}_{p-1}^p-1=-\zetax p^{1+\frac{1}{p-1}}\cdot U_p+O\left(p^{2+\frac{2}{p-1}}\right).\]
    Since $$v_p\left(\left(\zetax p^{\frac{1}{p-1}}U_p+A-\tilde{\Lambda}_{p-1}^++O\left(p^{1+\frac{2}{p-1}}\right)\right)^t\right)\geq t\left(1+\frac{1}{p-1}\right)$$
    and $\widehat{\tilde{\Lambda}}_{p-1}^{p-t}=1+O\left(p^{\frac{1}{p-1}}\right)$,
    we have
    \begin{align*}
        A^p-1= & -\zetax p^{1+\frac{1}{p-1}}\cdot U_p+O\left(p^{2+\frac{2}{p-1}}\right)                                                                                    \\
               & \quad+p\left(\zetax p^{\frac{1}{p-1}}U_p+A-\tilde{\Lambda}_{p-1}^++O\left(p^{1+\frac{2}{p-1}}\right)\right)\left(1+O\left(p^{\frac{1}{p-1}}\right)\right) \\
        =      & p\left(A-\tilde{\Lambda}_{p-1}^+\right)+O\left(p^{2+\frac{2}{p-1}}\right).
    \end{align*}
\end{proof}

\subsubsection{The truncated expansions of $\Lambda^p$ }
\begin{enumerate}[leftmargin=\itemindent]
    \item We first study the $\left(1+\frac{1}{p-1}+\frac{2}{p(p-1)}-\frac{1}{p^2}\right)$-truncated expansion of
          $$\Lambda\coloneqq\Lambda_{p-1}+\left(1+\lambda\right)\zetax p^{\frac{1}{p-1}}\sigma_2$$
          with $\sigma_2=\sum_{k=2}^{+\infty}p^{-1/p^k}$.
          \begin{lemma}\label{lem:23648}
              The $\left(1+\frac{1}{p-1}+\frac{2}{p(p-1)}-\frac{1}{p^2}\right)$-truncated expansion of $\Lambda^p-1$ is the following:
              $$\widehat{\tilde{\Lambda}}_{p-1}-1+\zetax p^{\frac{1}{p-1}}\cdot U_p+\zetax^2 p^{1+\frac{1}{p-1}+\frac{1}{p(p-1)}}+O\left(p^{1+\frac{1}{p-1}+\frac{2}{p(p-1)}-\frac{1}{p^2}}\right).$$
          \end{lemma}
          \begin{proof}
              Since $\Lambda=\Lambda_{p-1}+\left(1+\zetax p^{\frac{1}{p(p-1)}}\right)\zetax p^{\frac{1}{p-1}}\sigma_2$, we have
              \[ \Lambda^p-1= \Lambda_{p-1}^p-1+\sum_{j=1}^{p}\binom{p}{j}\Lambda_{p-1}^{p-j}\left(\left(1+\zetax p^{\frac{1}{p(p-1)}}\right)\zetax p^{\frac{1}{p-1}}\sigma_2\right)^j.   \]
              Notice that, the \Cref{lemma:12551} shows that $\sigma_2^p=\sigma_1+O(p^{1-1/p})$. Thus,
              \begin{enumerate}
                  \item for $2\leq j\leq p-1$, we have
                        \begin{align*}
                            v_p\left(\binom{p}{j}\Lambda_{p-1}^{p-j}\left(\left(1+\zetax p^{\frac{1}{p(p-1)}}\right)\zetax p^{\frac{1}{p-1}}\sigma_2\right)^j\right)
                            \geq  1+ \frac{2}{p-1}-\frac{2}{p^2}.\end{align*}
                  \item for $j=p$, by binomial expansion and truncation, we have
                        \begin{align*}
                              & \left(\left(1+\zetax p^{\frac{1}{p(p-1)}}\right)\zetax p^{\frac{1}{p-1}}\sigma_2\right)^p                                                                                                                                                                 \\
                            = & \left(1-\zetax p^{\frac{1}{p-1}}\right)                                                                                         \cdot\left(-\zetax p^{1+\frac{1}{p(p-1)}}-\zetax p^{1+\frac{1}{p-1}}\sigma_2+O\left(p^{2+\frac{1}{p(p-1)}}\right)\right).
                        \end{align*}
                  \item for $j=1$, by binomial expansion and truncation, we have
                        \begin{align*}
                              & p\Lambda_{p-1}^{p-1}\left(\left(1+\zetax p^{\frac{1}{p(p-1)}}\right)\zetax p^{\frac{1}{p-1}}\sigma_2\right)                                                                                \\
                            = & \left(1-\zetax p^{\frac{1}{p(p-1)}}\right)\left(1+\zetax p^{\frac{1}{p(p-1)}}\right)\zetax p^{1+\frac{1}{p-1}}\sigma_2 +O\left(p^{1+\frac{1}{p-1}+\frac{2}{p(p-1)}-\frac{1}{p^2}}\right) .
                        \end{align*}
              \end{enumerate}
              As a consequence, we have
              \begin{align*}
                  \Lambda^p-1
                  =  \Lambda_{p-1}^p-1-\zetax p^{1+\frac{1}{p(p-1)}}+\zetax^2 p^{1+\frac{1}{p-1}+\frac{1}{p(p-1)}}                                                  +O\left(p^{1+\frac{1}{p-1}+\frac{2}{p(p-1)}-\frac{1}{p^2}}\right).
              \end{align*}

              By \Cref{lem:43810}, we have
              \begin{align*}
                  \Lambda_{p-1}^p-1=  \widehat{\tilde{\Lambda}}_{p-1}-1+\zetax p^{1+\frac{1}{p(p-1)}}+\zetax p^{\frac{1}{p-1}}\cdot U_p +O\left(p^{1+\frac{1}{p-1}+\frac{2}{p(p-1)}-\frac{1}{p^2}}\right).
              \end{align*}
              Therefore,
              \begin{align*}
                  \Lambda^p-1=  \widehat{\tilde{\Lambda}}_{p-1}-1+\zetax p^{\frac{1}{p-1}}\cdot U_p+\zetax^2 p^{1+\frac{1}{p-1}+\frac{1}{p(p-1)}} +O\left(p^{1+\frac{1}{p-1}+\frac{2}{p(p-1)}-\frac{1}{p^2}}\right).
              \end{align*}
          \end{proof}
          By applying \Cref{prop:55108} to $\Lambda^p$, we know that
          \begin{equation}\label{eq:31503}
              \Lambda^{p^2}-1=\zetax^2 p^{2+\frac{1}{p-1}+\frac{1}{p(p-1)}}+o\left(p^{2+\frac{1}{p-1}+\frac{1}{p(p-1)}}\right).
          \end{equation}
    \item Now, we turn to study the second case. Recall that, we set
          \begin{align*}
              \mu_0=\Lambda_{p-1}+\left(1+\lambda\right)\zetax p^{\frac{1}{p-1}}\sigma_2, r_0=\frac{1}{p-1}+\frac{1}{p(p-1)}, r=\frac{2}{p-1}
          \end{align*} and
          $\mu=\Lambda_{p-1}\left(1+\zetax p^{\frac{1}{p-1}}\sigma_2\right)$.
          Suppose $\Lambda=\mu+\calM$ is a $r$-truncated expansion of $\zeta_{p^2}$ with an indeterminate $\calM$. By applying the $r_0$-truncation $\Lambda=\mu_0+O(p^{r_0})$ of $\Lambda$, we have
          \[
              \Lambda^p
              =         \tilde{\Lambda}_{p-1}^++o\left(p^{1+\frac{1}{p-1}}\right),
          \]
          where $ \tilde{\Lambda}_{p-1}^+=\sum_{l=0}^{p-1}\frac{(-1)^l}{[l!]}\zetax^l p^{\frac{l}{p-1}}+\zetax p^{\frac{1}{p-1}}U_p$.
          The following lemma confirms the assertion \Cref{cote1}.
          \begin{lemma}\label{lem:26652}
              We have $\Lambda^p=\tilde{\Lambda}_{p-1}^++O\left(p^{1+r}\right)$.
          \end{lemma}
          \begin{proof}Since $ \Lambda^p
                  =         \tilde{\Lambda}_{p-1}^++o\left(p^{1+\frac{1}{p-1}}\right)$, by \Cref{prop:55108} we have
              \begin{equation}\label{eq:60660}
                  \Lambda^{p^2}-1=p\left(\Lambda^p-\tilde{\Lambda}_{p-1}^+\right)+O\left(p^{2+\frac{2}{p-1}}\right).
              \end{equation}

              On the other hand, since $\Lambda$ is an approximation of $\zeta_{p^2}$, the maximal slope of the Newton polygon of $\Phi_{p^2}\left(T+\Lambda \right)$ is equal to or greater than $ r=\frac{2}{p-1}$.

              Denote by $\calP=(p(p-1)-1,y_\calP)$ the intersection of the line $x=p(p-1)-1$ and the segment $\calL$ in $\Newt{\Phi(T+\Lambda)}$ with the maximal slope. By \cite[Lemma 2.4]{WangYuan2021}, the endpoints of $\calL$ are on or above the Newton polygon $\Newt{\Phi^{(k,2)}}$, where $\Phi^{(k,2)}=\Phi_{p^2}\left(T+\zeta_{p^2}^{(k-1)}\right)$, for any integer number $k\geq p$. Therefore, the point $\calP$ is on or above $\Newt{\Phi^{(k,2)}}$ for any $k\in\bbZ_{\geq p}$. By \cite[Fig. 3]{WangYuan2021}, for every $k\in\bbZ_{\geq p}$, we can calculate that the line $x=p(p-1)-1$ intersects with the segment with maximal slope in $\Newt{\Phi^{(k,2)}}$ at $\left(p(p-1)-1,2-\frac{1}{p-1}-\frac{p-1}{p^{k-p+2}}\right)$. As a consequence, for $k\in\bbZ_{\geq p}$, we have  $y_\calP\geq 2-\frac{1}{p-1}-\frac{p-1}{p^{k-p+2}}$, which implies $y_\calP\geq 2-\frac{1}{p-1}$. Therefore the maximal slope $\fraks_{\Lambda}=v_p\left(\frac{\Lambda^{p^2}-1}{\Lambda^p-1}\right)-y_\calP$ of $\Newt{\Phi(T+\Lambda)}$ is
              $\geq \frac{2}{p-1}$,
              and consequently
              \begin{align*}
                  v_p\left(\Lambda^{p^2}-1\right)\geq  v_p\left(\Lambda^p-1\right)+y_\calP+\frac{2}{p-1}
                  \geq            2+\frac{2}{p-1}.
              \end{align*}

              Compare this inequality with \Cref{eq:60660}, we know that
              $$v_p\left(p\left(\Lambda^p-\tilde{\Lambda}_{p-1}^+\right)\right)\geq 2+\frac{2}{p-1}.$$

          \end{proof}
\end{enumerate}

\subsubsection{Solve the congruence equation}\label{truncated3}
Recall that, we set
\begin{align*}
    \mu_0=\Lambda_{p-1}+\left(1+\lambda\right)\zetax p^{\frac{1}{p-1}}\sigma_2, r_0=\frac{1}{p-1}+\frac{1}{p(p-1)}, r=\frac{2}{p-1}
\end{align*} and
$\mu=\Lambda_{p-1}\left(1+\zetax p^{\frac{1}{p-1}}\sigma_2\right)$.
Suppose $\Lambda=\mu+\calM$ is a $r$-truncated expansion of $\zeta_{p^2}$ with an indeterminate $\calM$.
In the following, we compute the $(1+r)$-truncated expansion of $\mu\left(\frac{\tilde{\Lambda}_{p-1}^+}{\mu^p}-1\right)$ and solve the equation $$p\calM=\mu\left(\frac{\tilde{\Lambda}_{p-1}^+}{\mu^p}-1\right)+O\left(p^{1+r}\right).$$
\begin{lemma}\label{lem:23892} Let   \begin{align*}
        \calW= & \sum_{l=0}^{p-1}\frac{(-1)^l}{l!}\zetax^l p^{\frac{l}{p-1}}+\zetax p^{\frac{1}{p-1}}U_p+\zetax p^{1+\frac{1}{p(p-1)}} \\
               & \quad-\zetax^2 p^{1+\frac{1}{p(p-1)}+\frac{1}{p-1}}+\frac{1}{2}\zetax^2 p^{1+\frac{2}{p-1}}\sigma_2+\newerr.
    \end{align*}
    Then we have
    $$\frac{\tilde{\Lambda}_{p-1}^+}{\mu^p}=\frac{\calW}{\Lambda_{p-1}^p}+\newerr.$$

\end{lemma}
\begin{proof}
    By the definition of $\mu$, we only need to prove the truncated expansion of
    $$\frac{\tilde{\Lambda}_{p-1}^+}{\left(1+\zetax p^{\frac{1}{p-1}}\sigma_2\right)^p}+\newerr$$
    is $\calW$.
    By binomial expansion,
    $$\left(1+\zetax p^{\frac{1}{p-1}}\sigma_2\right)^p= 1+\zetax^p p^{\frac{p}{p-1}}\sigma_2^p+\sum_{k=1}^{p-1}\binom{p}{k}\zetax^k p^{\frac{k}{p-1}}\sigma_2^k .$$
    Since $v_p\left(\zetax^k p^{\frac{k}{p-1}}\sigma_2^k\right)\geq \frac{2}{p-1} \text{ for }k\geq 3$, and by \Cref{lemma:12551}, we have
    $$\sigma_2^p-\sigma_2=p^{-1/p}+\sigma_2+O\left(p^{1-1/p}\right)-\sigma_2=p^{-1/p}+O\left(p^{1-1/p}\right),$$
    we can write
    \begin{equation}\label{anonying2}
        \begin{split}
            \left(1+\zetax p^{\frac{1}{p-1}}\sigma_2\right)^p=
            & 1-\zetax p^{1+\frac{1}{p-1}}\left(\sigma_2^p-\sigma_2\right)              -\frac{1}{2}\zetax^2 p^{1+\frac{2}{p-1}}\sigma_2^2+\newerr \\
            =                 & 1-\zetax p^{1+\frac{1}{p(p-1)}}-\frac{1}{2}\zetax^2 p^{1+\frac{2}{p-1}}\sigma_2^2+\newerr.
        \end{split}
    \end{equation}

    By taking multiplicative inverse and expanding into power series, we have
    \begin{align*}
        \left(1+\zetax p^{\frac{1}{p-1}}\sigma_2\right)^{-p}= 1+\zetax p^{1+\frac{1}{p(p-1)}}+\frac{1}{2}\zetax^2 p^{1+\frac{2}{p-1}}\sigma_2^2+\newerr.
    \end{align*}
    As a result, by expanding the product and using the congruence $[k]\equiv k \bmod{p}$, the truncated expansion of $\frac{\tilde{\Lambda}_{p-1}^+}{\left(1+\zetax p^{\frac{1}{p-1}}\sigma_2\right)^p}$ is the following :
    \begin{align*}
          & \tilde{\Lambda}_{p-1}^+\left(1+\zetax p^{1+\frac{1}{p(p-1)}}+\frac{1}{2}\zetax^2 p^{1+\frac{2}{p-1}}\sigma_2^2+\newerr\right) \\
        = & \sum_{l=0}^{p-1}\frac{(-1)^l}{l!}\zetax^l p^{\frac{l}{p-1}}+\zetax p^{\frac{1}{p-1}}U_p+\zetax p^{1+\frac{1}{p(p-1)}}         \\
          & \quad-\zetax^2 p^{1+\frac{1}{p(p-1)}+\frac{1}{p-1}}+\frac{1}{2}\zetax^2 p^{1+\frac{2}{p-1}}\sigma_2^2+\newerr,
    \end{align*}
    which is exactly $\calW+\newerr$.
\end{proof}
Recall we set $\kappa= -\sum_{n=2}^{p-1}\frac{(-1)^{n+1}}{n!n}\zetax^{n+1}p^{1+\frac{1}{p-1}+\frac{n}{p(p-1)}}.$
\begin{corollary}\label{coro:46486}
    We have
    $$\frac{\tilde{\Lambda}_{p-1}^+}{\mu^p}-1=-\zetax^2 p^{1+\frac{1}{p(p-1)}+\frac{1}{p-1}}+\frac{1}{2}\zetax^2 p^{1+\frac{2}{p-1}}\sigma_2^2+\kappa+\newerr.$$
\end{corollary}
\begin{proof} Using $ \frac{\tilde{\Lambda}_{p-1}^+}{\mu^p}-1= \calW\cdot\frac{1}{\Lambda_{p-1}^p}-1                                                                                                 $, the formula for $\calW$ (cf. \Cref{lem:23892} ) and the $O\left(p^{1+\frac{2}{p-1}}\right)$-truncated expansion of $\frac{1}{\Lambda_{p-1}^p}$(cf. \Cref{lem:2239}), we have
    \begin{equation}\label{anonying3}
        \begin{split}
            \frac{\tilde{\Lambda}_{p-1}^+}{\mu^p}-1
            =& -1-\zetax^2 p^{1+\frac{1}{p(p-1)}+\frac{1}{p-1}}+\frac{1}{2}\zetax^2 p^{1+\frac{2}{p-1}}\sigma_2^2                                             \\
            & \quad +\kappa  +\left(\sum_{i=0}^{p-1}\frac{(-1)^i}{i!}\zetax^l p^{\frac{i}{p-1}}\right)\left(\sum_{j=0}^{p-1} \frac{\zetax^j}{j!}p^{\frac{j}{p-1}}\right)+\newerr.
        \end{split}
    \end{equation}

    The product $\left(\sum_{i=0}^{p-1}\frac{(-1)^i}{i!}\zetax^i p^{\frac{i}{p-1}}\right)\left(\sum_{j=0}^{p-1} \frac{\zetax^j}{j!}p^{\frac{j}{p-1}}\right)$ can be expanded as
    \[  \sum_{m=0}^{p-1}\frac{1}{m!}\zetax^m p^{\frac{m}{p-1}}\sum_{i=0}^m\binom{m}{i}(-1)^i+\newerr.
    \]
    Since
    $$\sum_{i=0}^m\binom{m}{i}(-1)^i=\begin{cases}1,&\ m=0\\0,&\ m>0\end{cases},$$
    we know that
    $$\left(\sum_{i=0}^{p-1}\frac{(-1)^i}{i!}\zetax^l p^{\frac{i}{p-1}}\right)\left(\sum_{j=0}^{p-1} \frac{\zetax^j}{j!}p^{\frac{j}{p-1}}\right)=1+\newerr.$$
    As a result, we have
    $$\Cref{anonying3}=-\zetax^2 p^{1+\frac{1}{p(p-1)}+\frac{1}{p-1}}+\frac{1}{2}\zetax^2 p^{1+\frac{2}{p-1}}\sigma_2^2+\kappa+\newerr.$$
\end{proof}
\begin{lemma}\label{lem:38120}
    One has
    \begin{align*}
        \calM= & \frac{1}{2}\zetax^2 p^{\frac{2}{p-1}}\sigma_2^2+\frac{1}{2}\zetax^3 p^{\frac{2}{p-1}-\frac{p-2}{p^2(p-1)}} -\sum_{k=1}^{p-1}\frac{H_k}{k!}\zetax^{k+1} p^{\frac{1}{p-1}+\frac{k}{p(p-1)}} +O\left(p^{\frac{2}{p-1}}\right).
    \end{align*}
\end{lemma}
\begin{proof}Recall $\calM=\frac{1}{p}\left(1+\zetax p^{\frac{1}{p-1}}\sigma_2\right)\Lambda_{p-1}\left(\frac{\tilde{\Lambda}_{p-1}^+}{\mu^p}-1\right)$.
    By \Cref{coro:46486}, after expanding the product and rearranging terms, we obtain
    \begin{align*}
        \Lambda_{p-1}\left(\frac{\tilde{\Lambda}_{p-1}^+}{\mu^p}-1\right) = & \newerr+\frac{1}{2}\zetax^2 p^{1+\frac{2}{p-1}}\sigma_2^2+\frac{1}{2}\zetax^3 p^{1+\frac{2}{p-1}-\frac{p-2}{p^2(p-1)}}           \\
                                                                            & \quad-\zetax^2 p^{1+\frac{1}{p-1}+\frac{1}{p(p-1)}}                                                                              \\
                                                                            & \quad -\sum_{n=2}^{p-1}\frac{1}{n!}\zetax^{n+1}p^{1+\frac{1}{p-1}+\frac{n}{p(p-1)}}\sum_{k=1}^n\binom{n}{k}\frac{(-1)^{k+1}}{k}.
    \end{align*}

    By \cite[(9.2)]{Boyadzhiev2018}, we know that $\sum_{k=1}^n\binom{n}{k}\frac{(-1)^{k+1}}{k}=H_n$. Therefore by moving the term $-\zetax^2 p^{1+\frac{1}{p-1}+\frac{1}{p(p-1)}}$ into the summation,
    \begin{align*}
          & \Lambda_{p-1}\left(\frac{\tilde{\Lambda}_{p-1}^+}{\mu^p}-1\right)                                                                                            \\
        = & \frac{1}{2}\zetax^2 p^{1+\frac{2}{p-1}}\sigma_2^2+\frac{1}{2}\zetax^3 p^{1+\frac{2}{p-1}-\frac{p-2}{p^2(p-1)}}
        -\sum_{n=1}^{p-1}\frac{H_n}{n!}\zetax^{n+1}p^{1+\frac{1}{p-1}+\frac{n}{p(p-1)}}+\newerr.
    \end{align*}
    As a result,
    \begin{align*}
        \calM= & \frac{1}{p}\left(1+\zetax p^{\frac{1}{p-1}}\sigma_2\right)\Lambda_{p-1}\left(\frac{\tilde{\Lambda}_{p-1}^+}{\mu^p}-1\right) \\
        =      & \frac{1}{2}\zetax^2 p^{\frac{2}{p-1}}\sigma_2^2+\frac{1}{2}\zetax^3 p^{\frac{2}{p-1}-\frac{p-2}{p^2(p-1)}}
        -\sum_{n=1}^{p-1}\frac{H_n}{n!}\zetax^{n+1}p^{\frac{1}{p-1}+\frac{n}{p(p-1)}}+O\left(p^{\frac{2}{p-1}}\right).
    \end{align*}
\end{proof}
Finally, we obtain the $\frac{2}{p-1}$-truncated expansion of $\zeta_{p^2}$:
\begin{theorem}\label{mainexpansion}
    We have
    \begin{align*}
        \zeta_{p^2}= & \left(\sum_{k=0}^{p-1}\frac{1}{[k!]}\zetax^k p^{\frac{k}{p(p-1)}}\right)\left(1+\zetax p^{\frac{1}{p-1}}\sigma_2\right)+\frac{1}{2}\zetax^2 p^{\frac{2}{p-1}}\sigma_2^2        \\
                     & \quad+\frac{1}{2}\zetax^3p^{\frac{2}{p-1}-\frac{p-2}{p^2(p-1)}}-\sum_{k=1}^{p-1}\frac{H_k}{k!}\zetax^{k+1} p^{\frac{1}{p-1}+\frac{k}{p(p-1)}}+O\left(p^{\frac{2}{p-1}}\right).
    \end{align*}
\end{theorem}
\begin{remark}
    One obtains the $\frac{2}{p-1}$-truncated canonical expansion of $\zeta_{p^2}$ by expanding the product in \Cref{mainexpansion} and replace the coefficient of each term by its Teichmuller lifting.
\end{remark}

\subsection{The truncated expansion of $\zeta_{p^n}$}
\begin{lemma}\label{monter}
    Suppose $A\in\calO_{\bbL_p}$ has $1$-truncated canonical expansion
    $$A=\sum_{g\in\Supp{A}\cap [0,1)}[a_g]p^g +O\left(p^1\right).$$
    Then there exists a $p$-th root of $A$ in $\bbL_p$ with $\frac{1}{p}$-truncated canonical expansion given by
    $$B=\sum_{g\in\Supp{A}\cap [0,1)}[a_g^{\frac{1}{p}}]p^{\frac{g}{p}}+O\left(p^{\frac{1}{p}}\right)$$
\end{lemma}
\begin{proof}
    Let $B_0=\sum_{g\in\Supp{A}\cap [0,1)}[a_g^{\frac{1}{p}}]p^{\frac{g}{p}}$. We claim that $B_0$ is an approximation of a $p$-th root $A^{\frac{1}{p}}$ of $A$ in $\bbL_p$.

    Consider the polynomial $f(T)=\left(T+B_0\right)^p-A\in \bbL_p[T]$. We estimate the $p$-adic valuation of the coefficients of $f(T)$. For $1\leq i\leq p-1$, we have $v_p\left(\binom{p}{i}B_0^{p-i}\right)\geq 1$.
    By \Cref{prop:16960}, one has
    $$B_0^p=\sum_{g\in\Supp{A}\cap[0,1)}[a_g]p^g+O\left(p^{1+v_p(B_0)}\right),$$
    and consequently $v_p\left(B_0^p-A\right)\geq 1$. Thus, except that the coefficient of leading term of $f(T)$ has valuation $0$, the coefficients of other terms in $f(T)$ are of $p$-adic valuation at least $1$. As a result, the maximal slope of $\Newt{f}$ is at least $\frac{1}{p}$. Therefore there exists a root $U\in\bbL_p$ of $f(T)$ with $v_p(U)\geq \frac{1}{p}$. Set $B=B_0+U$. Then we have $B^p=A$ and
    $B=\sum_{g\in\Supp{A}\cap [0,1)}[a_g^{\frac{1}{p}}]p^{\frac{g}{p}}+O\left(p^{\frac{1}{p}}\right).$

\end{proof}
Form this lemma and the $\frac{2}{p-1}$-truncated expansion of $\zeta_{p^2}$ (cf. \Cref{mainexpansion}), we can deduce the $\frac{2}{p^{n-2}(p-1)}$-truncated expansion of $\zeta_{p^n}$:
\begin{proposition}\label{truncatedfinal}
    For $n\geq 2$, we have the following $\frac{2}{p^{n-2}(p-1)}$-truncated expansion of $\zeta_{p^n}$:
    \begin{align*}
        \zeta_{p^n}= & \sum_{k=0}^{p-1}\frac{(-1)^{nk}}{k!}\zetax^k p^{\frac{k}{p^{n-1}(p-1)}}+\sum_{k=0}^{p-1}\frac{(-1)^{n(k+1)}}{k!}\zetax^{k+1}p^{\frac{k+p}{p^{n-1}(p-1)}}\sigma_n \\
                     & \quad-\sum_{k=1}^{p-1}\frac{H_k}{k!}(-1)^{n(k+1)}\zetax^{k+1} p^{\frac{k+p}{p^{n-1}(p-1)}}                                                                       \\
                     & \quad+\frac{1}{2}\zetax^2 p^{\frac{2}{p^{n-2}(p-1)}}\sigma_n^2+\frac{(-1)^n}{2}\zetax^3 p^{\frac{2}{p^{n-2}(p-1)}-\frac{p-2}{p^n(p-1)}}                          \\
                     & \quad+O\left(p^{\frac{2}{p^{n-2}(p-1)}}\right).
    \end{align*}
\end{proposition}
\begin{remark}
    We obtain the $\frac{2}{p^{n-2}(p-1)}$-truncated \textbf{canonical} expansion of $\zeta_{p^n}$ from \Cref{truncatedfinal} immediately by expanding the product and replacing the coefficients by their Teichmuller lifting. This is based on the fact :
    $$[k]=\begin{cases}
            k,      & \text{ if }k=0,1;   \\
            k+O(p), & \text{ if }k\geq 2.
        \end{cases}$$
    Taking account of the support of $\zeta_{p^n}$, we know that in general\footnote{There exists some exceptional primes such that the assertion is not true.}, for an odd prime $p$,  taking off or adding on the Teichmuller lifting does not affect the correctness of the truncated expansion if and only if the truncated position is not greater than $1+\frac{2}{p^{n-1}(p-1)}$. Thus, the Teichmuller lifting in the calculation may not be neglectable anymore for $r\geq 1+\frac{2}{p^{n-1}(p-1)}$. For example, $\widehat{\Lambda}_{p-1}$ is no longer a truncated expansion of $\Lambda_{p-1}$ when the truncated position is greater than $1+\frac{2}{p(p-1)}$. On the other hand, since combinatorial techniques are our primary tools, the truncated (noncanonical) expansion without the Teichmuller liftings is more convenient for us.
\end{remark}
\printbibliography
\end{document}